\documentclass[10pt]{amsart}

\usepackage{amssymb}
\usepackage{amsfonts,amsthm,amsmath}

\usepackage{graphicx,epstopdf}
\usepackage{setspace}
\usepackage{color}
\usepackage{xcolor}
\usepackage{multirow}
\usepackage{rotating}
\usepackage{float}
\usepackage{mathtools}
\def\E{{\mathbb{E}}}
\def\Eq{{\mathbb{E}
		_{x,\hat{x}}}}
\def\P{{\mathbb{P}}}
\def\Q{{\mathbb{Q}}}

\def \I{{\mathbb{I}}}
\def \F{{\mathcal{F}}}

\newcommand{\dx}{\operatorname{d}\!}

\usepackage[english]{babel}
\usepackage[left=3cm,right=3cm,top=2.5cm,bottom=3cm]{geometry}

\usepackage{amsmath,amssymb,amsthm,bbm,color,graphics,version}
\usepackage{mathrsfs}

\newcommand{\bearno}{\begin{eqnarray*}}
	\newcommand{\enarno}{\end{eqnarray*}}

 \def\Xi{X^{\infty}}

\def\P{{\mathbb P}}   
\def\E{{\mathbb E}}   
 \def\Q{{\mathbb Q}}

\newtheorem{Thm}{Theorem}

\newtheorem{Lemma}{Lemma}
\newtheorem{Prop}{Proposition}

\newtheorem{Rem}{Remark} 
 
\title{PRICING AMERICAN OPTIONS TIME-CAPPED BY A DRAWDOWN EVENT IN A L\'EVY MARKET}

\author{Pawe\l\ St\c{e}pniak}
\address{Faculty of Pure and Applied Mathematics, Wroc\l aw University of Science and Technology, Wroc\l aw, Poland} \email{pawel.stepniak@pwr.edu.pl}
\author{Zbigniew Palmowski}
\address{Faculty of Pure and Applied Mathematics, Wroc\l aw University of Science and Technology, Wroc\l aw, Poland} \email{zbigniew.palmowski@pwr.edu.pl}

\thanks{
	This work is partially supported by National Science Centre, Poland, under grants
	No. 2021/41/B/HS4/00599.}

\date{\today}
\begin{document}
	\begin{abstract}
		This paper presents a derivation of the explicit price for the perpetual American put option time-capped by the first drawdown epoch beyond a predefined level.
		We consider the market in which an asset price is described by geometric L\'evy process with downward exponential jumps.
		We show that the optimal stopping rule is the first time when the asset price gets below a special value.
		The proof relies on martingale arguments and the fluctuation theory of L\'evy processes. We also provide a numerical analysis.
		\vspace{3mm}
		

	\end{abstract}
	
	\maketitle

	\pagestyle{myheadings} \markboth{\sc P.\ St\c{e}pniak --- Z.\ Palmowski} {\sc Time-capped American options}

	\vspace{1.8cm}
	
	\tableofcontents
	
	\newpage
	
	\section{Introduction}
	\subsection{Main results}

	In mathematical finance, there has been growing interest in extending classical American options to include more complex financial products.
	In this paper, we consider the perpetual American put option time-capped by the first drawdown epoch beyond a predefined level.
	Such instrument should appeal to investors due to their reduced liability and lower cost compared to standard options. Popularity of this type of financial instrument strongly depends though on
	understanding their pricing, hedging, and optimal exercise policies. The desire to understand these features was our motivation for this paper.
	
	This type of financial instrument compliments the capped options involving a cap to the asset price, where the option is terminated if the asset value exceeds a specified threshold.
	One of the simplest examples of capped options was introduced in 1991 by the Chicago Board of Options Exchange European options written on the S\&P 100 and S\&P 500 with a cap on their payoff function (see \cite{cap}).
	For other papers related to this type of cap, see \cite{Zaevski2, Zaevski2b} and references therein.
	
	This type of financial instrument is also strongly related to the deterministic cap (hence American option with finite maturity) and
	with other time-caps like the first passage time over given level of the asset price.
	
	In this paper, consider the asset price $S_t$ in the L\'evy-type market, that is, we assume that under the risk-neutral measure,
	\begin{equation*}
		S_t = e^{X_t},
	\end{equation*}
	with
	%$X_t$ being a spectrally negative L\'evy process and $s = S_0$ representing the initial asset price. Specifically, we define
	\begin{equation}\label{X_def}
		X_t = x + \mu t + \sigma B_t  - \sum_{k=1}^{N_t}U_k, %\quad \forall_{k}\ U_k \geq 0,
	\end{equation}
	where $x = X_0 = \log s$ and $\sigma \geq 0$. In equation (\ref{X_def}), %$\mu$ is a fixed drift,
	$B_t$ denotes a Brownian motion, $N_t$ is a homogeneous Poisson process with intensity $\lambda$ and $\{U_k\}_{\{k\in \mathbb{N}\}}$ is a sequence of independent identically distributed exponential random variables with mean $\rho^{-1}$.
	We assume that $B_t$, $N_t$ and $\{U_k\}_{\{k\in \mathbb{N}\}}$ are mutually independent and that all processes considered live in a common filtered probability space $(\Omega, \F, \{\mathcal{F}_t\}_{\{t\geq 0\}}, \Q)$
	with natural filtration $\{\F_t\}_{\{t \geq 0\}}$ of $X_t$ satisfying the usual conditions. We allow $\lambda = 0$, which corresponds to the standard Black-Scholes model.
	For simplicity, we assume that no dividends are paid to the holders of the underlying asset.
	
	Although geometric Brownian motion hence the Black-Scholes market widely used in everyday applications is a special case of our model,
we add the possibility of the negative jumps in the asset price to better
fit the evolution of the stock price process to real data.
	In recent years, the empirical study of financial data reveals that the distribution
of the log-return of stock price exhibits features which cannot be captured by the normal
distribution such as heavy tails and asymmetry. With a view to replicate these features
more effectively and to reproduce a wide variety of implied volatility skews and smiles,
there has been a general shift in the literature to model a risky asset with a geometric
	L\'evy process analyzed in this paper rather than a geometric Brownian motion, see \cite{Schoutens} for details.
	We still keep in mind that there is no uniqe martingale measure in this set-up, so one can choose, for example, one minimizing the entropy; see e.g. \cite[Chap. 10.5]{Cont} for details and discussion.
	
	By
	\begin{equation*}
		\overline{S}_t = e^{\overline{x}}\vee \sup_{0 \leq u \leq t} S_u
	\end{equation*}
	we denote the running maximum of the asset price where
	$e^{\overline{x}}$ is the historical maximum of the underlying asset price prior to the beginning of the contract, where $d\vee f = \max\{d,f\}$. Similarly, we will denote $d\wedge f = \min\{d,f\}$.
	For the fixed threshold $c>0$ let
	\begin{equation*}
		\tau_D = \inf \left\{ t \geq 0: \frac{\overline{S}_t}{S_t} \geq e^c \right\}
	\end{equation*}
	be the first time when the (relative) drawdown is greater than $e^c$.
	In the main result of this paper we identify the closed-form formula for the price of the American put option with the random maturity determined to be a drawdown event given by
	\begin{equation*}
		V(x, \overline{x}) = \sup_{\tau \in \mathcal{T}} \E_{x,\overline{x}} \left[e^{-r\tau\wedge\tau_D}(K - S_{\tau\wedge\tau_D})^+\right],
	\end{equation*}
	for a family of stopping times $\mathcal{T}$ and fixed strike price $K > 0$.
	Above, the subindex $x,\overline{x}$ attached to the expectation $\E_{x,\overline{x}}$ underlines the dependence of the mean on
	the initial asset price $S_0=e^x$ and the historical (observed) supremum $\overline{S}_0=e^{\overline{x}}$. We will skip the subindex when $x=\overline{x}=0$.
	The expectation is taken with respect of measure $\Q$.
	When $x=0$, we will skip this index. We restrict the domain of $V$ to $\mathcal{D} := \{ (x, \overline{x}) \in \mathbb{R}^2 : x \leq \overline{x}\}$.
	In the main result, we also find the optimal exercise rule, that is, a stopping rule $\tau^*$ such that
	\begin{equation*}
		V(x, \overline{x}) = \E_{x,\overline{x}} [e^{-r\tau^*\wedge\tau_D}(K - S_{\tau^*\wedge\tau_D})^+].
	\end{equation*}
	Our main result is the following theorem.
	\begin{Thm}\label{thm2}
		The optimal stopping barrier is the first downward asset price time
		\begin{equation*}
			\tau^*=\inf\left\{t\geq 0: X_t\leq a^*\right\},
		\end{equation*}
		where $a^*$ is the unique solution of equation \eqref{otimala*}.
		%That is,   the option should be stopped at $X_t = a$ or $X_t = \overline{X_t}-c$, whichever occurs first.
		Moreover, we have $V(x, \overline{x})=V_{a^*}(x, \overline{x})$
		for the function $V_{a^*}(x, \overline{x})$ identified in  Theorem \ref{valuea}.
	\end{Thm}
	
	\subsection{Literature overview}
	Random termination appears in the mathematical finance literature, for example, as the default time of a company (see, e.g., \cite[p. 27]{BieleckiRutkowski} or \cite{Mladen}) or
	as an asset-price-independent time cap following an exponential or Erlang distribution (see \cite{carr, Florin}. In some cases, the time cap is chosen to be unobservable; see \cite{gapeev1}.
	
	Time-cap is closely related to cancellable options, which are terminated early when a specific event occurs. In fact, the price of the stopped option consists of the value of the cancellable option and the discounted payoff (under the risk-neutral measure) at the event time, provided the event happens before maturity.
	Typically, this event is described as the first or last time when the underlying asset price
	reaches a specific threshold; see, e.g., \cite{1st_pass, algo} and references therein.
	Similar early termination features can be found in game options (like, e.g., Israeli options) where the seller has the right to terminate the contract early, subject to a fixed penalty paid to the buyer; see the seminal paper \cite{kifer}
	and subsequent works such as \cite{meyer, Zaevski2, Zaevski2b}.
	
	There are other studies addressing derivatives that closely resemble those analyzed in this paper.
	Egloff, Farkas and Leippold in \cite{time_constraints} price American options with
	stochastic stopping time constraints, where the exercise of the option is restricted to specific conditions being met.
	To some extent, our motivation for considering a time-capped option is very similar:
	we want to introduce an option that remains valid until is either terminated by the buyer or triggered by an event described above, whichever occurs first.
	However, the difference in pricing is crucial. In \cite{time_constraints} the buyer can only exercise the option
when a pre-specified condition, associated with the performance of the underlying asset, is satisfied and when it is possible to transform the constrained pricing problem into an unconstrained optimal stopping problem.
	% that corresponds to a generalised barrier option pricing and a stochastic Cauchy-Dirichlet problem.
	For the time-capped American option considered in this paper, this idea cannot be realized.

	Several authors have also explored the concept of time-capping in American options.
	In \cite{trabelsi}, the random cap is a~first hitting time of a fixed barrier by
	the underlying asset price. Finally, there are other related works.
	For example, \cite{russian, random_put, ott} investigate Russian options that terminate when the stock price hits its running maximum for the last time, as well as American options that terminate when the stock price reaches a prespecified level for the last time.

	%In our case, the American option terminates early if a drawdown event occurs. Specifically, this means the option is exercised either at a random stopping time or at the first time when the drawdown of the stock price exceeds a fixed threshold, whichever comes first.
	We choose the drawdown event in pricing our American option due to its importance for the financial markets aiming
	%Explicitly including protection against significant drawdowns in financial contracts is a common practice aimed
	at minimizing potential losses for the seller.
	The interest in the drawdown process has arisen after recent financial crises. It has also been used as dynamic
risk measure or measure of relative regret, for example, the time-adjusted measure of performance known as the Calmar
ratio. Thus, the drawdown process has become very
	important among practitioners.
	
	The list of papers addressing contracts that incorporate drawdown or drawup feature is quite long; see,
	e.g. \cite{CZH, GZ, MA, ZP&JT2, ZP&JT, drawdownup2, PV, Sorn, Vec1,Vec2, drawdownup1, olympia}
	and references therein.
	
	In \cite{naszapraca} we analyzed the same problem, but only for the Black-Scholes market. The arguments given in this article requires much more complex analysis, though.

	\subsection{Organisation of the paper}
	This paper is organized as follows. In the next section, we give the proof of the main result. The proof is based on choosing a stopping rule and calculating the option value based on it. Next, a verification step via HJB system ensures that the chosen rule is indeed optimal and the achieved price is fair.
	In Section \ref{sec:num} we present a numerical analysis.

	\section{Proof of main result}
	\subsection{Basic facts and verification lemma}
	The proof of the main result is based on the following
	key observations.
	Their proofs are the same as in \cite{naszapraca} but we add them here for the completeness of the arguments.
	
	\begin{Prop}\label{thm1}
		The optimal stopping time $\tau^*$ is of the following form:
		\begin{equation*}
			\tau^* = \inf \{ t \geq 0: X_t \leq b(\overline{X_t}) \}
		\end{equation*}
		for some function $b$.
	\end{Prop}
	\begin{proof}
		Let
		\begin{equation*}
			D = \{ (x,\overline{x})\in \mathbb{R}^2: V(x, \overline{x}) = (K - e^x)^+ \}.
		\end{equation*}
		Note that
		\begin{equation*}
			Z_t = (X_t,\overline{X}_t)
		\end{equation*}
		is a Feller process.
		By \cite[Thm. 2.7, p. 40 and (2.2.80), p. 49]{PS} it follows that $$\tau^*=\inf\{t\geq 0: X_t\in D\},$$
that is, $D$ a stopping region, where the option should be exercised immediately.
		Suppose there exist $(x,\overline{x}) \in D$. We additionally assume that $x < \log K$, otherwise the immediate payout is zero. Let $\tau_y$ and $\tau_x$ be the optimal stopping rule for starting point $(y, \overline{x})$ and $(x, \overline{x})$ respectively. Observe that if $y < x < \log(K)$, then for a chosen $c$, we have $\tau_D(y, \overline{x}) \leq \tau_D(x,\overline{x})$, where $\tau_D(x,\overline{x})$ denotes $\tau_D$ for the starting point $(x, \overline{x})$. Further, we have:
		\begin{align*}
			V(y, \overline{x}) - V(x, \overline{x}) &= \E_{y,\overline{x}}[e^{-r\tau_y\wedge\tau_D(y,\overline{x})}(K - e^{X_{\tau_y\wedge\tau_D(y,\overline{x})}})^+] - \E_{x,\overline{x}}[e^{-r\tau_x\wedge\tau_D(x,\overline{x})}(K - e^{X_{\tau_x\wedge\tau_D(x,\overline{x})}})^+] \\
			&\leq \E_{y,\overline{x}}[e^{-r\tau_y\wedge\tau_D(y,\overline{x})}(K - e^{X_{\tau_y\wedge\tau_D(y,\overline{x})}})^+] - \E_{x,\overline{x}}[e^{-r\tau_y\wedge\tau_D(y,\overline{x})}(K - e^{X_{\tau_y\wedge\tau_D(y,\overline{x})}})^+] \\
			& = \E[e^{-r\tau_y\wedge\tau_D(y,\overline{x})}(K - e^{y + X_{\tau_y\wedge\tau_D(y,\overline{x})}})^+] - \E[e^{-r\tau_y\wedge\tau_D(y,\overline{x})}(K - e^{x + X_{\tau_y\wedge\tau_D(y,\overline{x})}})^+] \\
			& \leq \E[e^{-r\tau_y\wedge\tau_D(y,\overline{x})}(K - e^{y + X_{\tau_y\wedge\tau_D(y,\overline{x})}})] - \E[e^{-r\tau_y\wedge\tau_D(y,\overline{x})}(K - e^{x + X_{\tau_y\wedge\tau_D(y,\overline{x})}})] \\
			& = (e^x - e^y)\E[e^{-r\tau_y\wedge\tau_D(y,\overline{x}) + X_{\tau_y\wedge\tau_D(y,\overline{x})}}] = e^x - e^y.
		\end{align*}
		Therefore we get
		\begin{equation*}
			V(y, \overline{x}) - V(x, \overline{x}) \leq  e^x - e^y = (K - e^y) - (K-e^x)
		\end{equation*}
		and further
		\begin{equation*}
			V(y, \overline{x}) \leq (K - e^y) \leq  (K - e^y)^+.
		\end{equation*}
		On the other hand the payoff function of the option cannot be higher than its value function, therefore
		\begin{equation*}
			V(y, \overline{x}) \geq  (K - e^y)^+.
		\end{equation*}
		This gives $(y, \overline{x}) \in D$. This leads to the conclusion that for a certain $\overline{x}$ the optimal stopping region can be achieved by the pair $(X_t, \overline{X_t})$ when $X_t$ drops down to some value $b(\overline{X_t})$ before it reaches its past maximum.
		
		The question remains whether the optimal stopping region can also be reached from below, when both $X_t$ and $\overline{X_t}$ hit a certain level for the first time. We will show that this scenario is not possible.
		Indeed, assume a contrario that there exists a threshold $b$ such that there exist $x$ and  $\overline{x} < b$ satisfying $(x, \overline{x}) \notin D$ and $(b,b) \in D$. Let us take a positive $\varepsilon$ such that $b - \varepsilon > \overline{x}$. Let $\vartheta$ be the first time that process $X_t$ reaches the level $b - \varepsilon$ from below, i.e.
		\begin{equation*}
			\vartheta = \inf \{ t > 0: X_t = b - \varepsilon \}
		\end{equation*}
		and let $\tau_b$ be the first time when process $X_t$  reaches the level $b$ from below. Clearly, we have $\vartheta < \tau_b$ and $e^{X_\vartheta} < e^{X_{\tau_b}}.$ Therefore
		\begin{equation*}
			\E_{x,\overline{x}}[e^{-r\vartheta\wedge\tau_D(x,\overline{x})}(K - e^{X_{\vartheta\wedge\tau_D(x,\overline{x})}})^+] \geq \E_{x,\overline{x}}[e^{-r\tau_b\wedge\tau_D(x,\overline{x})}(K - e^{X_{\tau_b\wedge\tau_D(x,\overline{x})}})^+],
		\end{equation*}
		%This implies it must dominate any alternative stopping rule in terms of the expected payoff. Inequality (\ref{contradiction})
		which contradicts the assumption and completes the proof.
		
	\end{proof}
	
	We will also need the following verification lemma.
	By
	\begin{align*}
		\mathcal{L}&f(x, \overline{x}) = \left(r - \frac{\sigma^2}{2}\right) \frac{\partial}{\partial x} f(x, \overline{x}) +  \frac{\sigma^2}{2}\frac{\partial^2}{\partial x^2} f(x, \overline{x})  \\
		+ & \lambda\rho\int\limits_0^\infty \left( f(x-y, \overline{x}) - f(x, \overline{x}) \right) e^{-\rho y}\dx y \quad \text{for}\quad 0 < x < \overline{x},
	\end{align*}
	we denote the generator of the Markov process $(X_t, \overline{X}_t)$
and the domain of this generator includes the functions $f\in \mathcal{C}^2_0(\mathbb{R})$ such that
	\begin{align}
		& \frac{\partial}{\partial \overline{x}} f(x, \overline{x})  =  0  \quad \text{for}\quad  x = \overline{x}. \label{gen_dom_cut}
	\end{align}	
	\begin{Lemma}\label{HJB}
		Let $\hat{V}(x, \overline{x}): \mathbb{R}^2 \rightarrow \mathbb{R}$ be a function defined on $\mathcal{D} := \{ (x, \overline{x}) \in \mathbb{R}^2 : x \leq \overline{x}\}$.
		Assume that $\hat{V}(x, \overline{x}) \in \mathcal{C}^2_0(\mathbb{R})$ and that it fulfills condition \eqref{gen_dom_cut}.
		Assume that for some function $b$,
		\begin{align}
			(\mathcal{L} \hat{V} - r\hat{V})(x, \overline{x}) &= 0\quad \text{for}\ x> b(\overline{x}), \label{HJB1} \\
			(\mathcal{L} \hat{V} - r\hat{V})(x, \overline{x}) &\leq 0\quad \text{for}\ x\leq  b(\overline{x}),\label{HJB1b}
		\end{align}
		\begin{align}
			\hat{V}(x, \overline{x}) &= (K - e^x)^+\quad \text{for}\ x \leq b(\overline{x}), \label{HJB3}\\
			\hat{V}(x, \overline{x}) &> (K - e^x)^+\quad \text{for}\ x> b(\overline{x}), \label{HJB4}
		\end{align}
		\begin{align}
			\hat{V}(x, \overline{x}) \big|_{x = b(\overline{x})} &= (K - e^{b(\overline{x})})^+, \label{HJB6}  \\
			\frac{\partial}{\partial x} \hat{V}(x, \overline{x}) \big|_{x = b(\overline{x})} &= \frac{\partial}{\partial x} (K - e^x)^+ \big|_{x = b(\overline{x})} \quad \text{if}\ b(\overline{x})<\overline{x}-c.\label{HJB7}
		\end{align}
		Then $\hat{V}(x, \overline{x})\geq V(x, \overline{x})$.
	\end{Lemma}
	\begin{Rem}
		\rm Conditions \eqref{HJB6} and \eqref{HJB7} are the so-called smooth paste conditions of the value function. Note that the condition \eqref{HJB6} is mainly required
		to write \eqref{HJB7} which is used in the proof of Lemma \ref{HJB}.
	\end{Rem}
	\begin{proof}
		Due to the assumed smoothness of $\hat{V}$, the smooth paste conditions \eqref{HJB6} - \eqref{HJB7} and an appropriate version It\^{o}'s theorem
		(see \cite[p. 208]{EisenbaumKyprianou}) we have
		\begin{align*}
			&e^{-rt}\hat{V}(X_t, \overline{X}_t) =
			\hat{V}(x,\overline{x}) + \sigma\int_0^t e^{-r u} \frac{\partial}{\partial x} \hat{V}(X_u, \overline{X}_u) \dx B_u + \int_0^t e^{-r u} \frac{\partial}{\partial \overline{x}} \hat{V}(X_u, \overline{X}_u) \dx \overline{X}_u \nonumber\\
			&\quad + \int_0^t e^{-r u} \left( \mathcal{L}\hat{V}(X_u, \overline{X}_u) - r\hat{V}(X_u, \overline{X}_u) \right) \dx u\\&\quad  +\frac{1}{2}
			\int_0^t \left( \frac{\partial}{\partial x} \hat{V}(x, \overline{x}) \big|_{x = b(\overline{x})} - \frac{\partial}{\partial x} (K - e^x)^+ \big|_{x = b(\overline{x})}\right)\dx L(s), \nonumber
		\end{align*}
		where $L$ is a local time of the process $X - b(\overline{X})$ at 0. Observe that $\dx \overline{X}_u = \I\{ X_u = \overline{X}_u \}\dx X_u$.
		
		Now, requirement (\ref{gen_dom_cut}) guarantees that the integral over $\dx \overline{X}_u$ is zero. Similarly, the smooth-paste conditions make sure that the integral over the local time also vanishes. Finally, relations (\ref{HJB1}) and (\ref{HJB1b}) lead to the conclusion that the integral over $\dx u$ is non-positive. If we take the expectation of both sides, we get the following result
		\begin{align*}
			e^{-rt}&\E_{x,\overline{x}}\hat{V}(X_t, \overline{X}_t) = \hat{V}(x,\overline{x}) + \sigma \E_{x,\overline{x}} \int_0^t e^{-r u} \frac{\partial}{\partial x} \hat{V}(X_u, \overline{X}_u) \dx B_u \\
			&+  \E_{x,\overline{x}} \int_0^t e^{-r u} \left( \mathcal{L}\hat{V}(X_u, \overline{X}_u) - r\hat{V}(X_u, \overline{X}_u) \right) \dx u \leq \hat{V}(x,\overline{x})
		\end{align*}
		since the integral over Brownian motion is a zero-mean local martingale. Hence, according to the assumptions made, the process $e^{-rt\wedge \tau_D}\hat{V}(X_{t\wedge \tau_D}, \overline{X}_{t\wedge \tau_D})=e^{-rt\wedge \tau_D}\hat{V}(Z_{t\wedge \tau_D})$ is a supermartingale and $\hat{V}(x,\overline{x})$ is a superharmonic function that dominates the payout. Now, from \cite[(2.2.80), p. 49]{PS} we additionally know that $\hat{V}$ is also lower semi-continuous. It allows us to use \cite[Thm. 2.7, p. 40]{PS} and claim that it is the optimal solution to the considered stopping problem. Hence $\hat{V}(x, \overline{x})\geq V(x, \overline{x})$.
		%and the statement follows from \cite[Thm. 2.7, p. 40 and (2.2.80), p. 49]{PS}.
	\end{proof}
	\begin{Rem}\label{uwaga2}
		\rm
		By Proposition \ref{thm1} we know that the optimal stopping rule $\tau^*$
		is of the one-sided form. In the next step, we postulate that the optimal
		stopping boundary is even more specific, namely, that
		\begin{itemize}
			\item $b(\overline{x})=a^*$ for some optimal $a^*$ when $\overline{x} < a^* + c$;
			\item $b(\overline{x})=\overline{x}-c$ when $a^* + c <  \overline{x} < \log(K)+c$.
		\end{itemize}
		We will calculate the value function
		\begin{equation*}
			\hat{V}(x, \overline{x})=\E_{x,\overline{x}} [e^{-r\tau^*\wedge\tau_D}(K - S_{\tau^*\wedge\tau_D})^+]
		\end{equation*}
		for this postulated stopping rule  $\tau^*$ and we will show that all the assumptions of Lemma \ref{HJB}
		are satisfied for $\hat{V}$. Hence in this case
		$\hat{V}(x, \overline{x})\geq V(x, \overline{x})$  by  Lemma \ref{HJB} and $\hat{V}(x, \overline{x})\leq V(x, \overline{x})$ due to the fact that
we choose a specific stopping rule. Thus, $\hat{V}(x, \overline{x})=V(x, \overline{x})$ is a true value function
and $\tau^*$ is the optimal stopping rule.

		From general stopping theory applied to the Markov process $(S_{t\wedge \tau_D}, \overline{S}_{t\wedge \tau_D})$
		(see \cite[Thm. 2.7, p. 40 and (2.2.80), p. 49]{PS}) it follows that the stopping region is the set when the value function meets the payout function
		and hence it is unique which is due to the existence of the value function
		$V(x, \overline{x})$. In other words, our stopping region is unique as well.
		
	\end{Rem}
	
	\subsection{Value function $V_a$}
	According to Remark \ref{uwaga2} we will first identify the value function
	\begin{equation*}%\label{Vaform}
		V_a(x, \overline{x}) = \E_{x,\overline{x}} \left[e^{-r\tau_a\wedge\tau_D}(K - S_{\tau_a\wedge\tau_D})^+\right],
	\end{equation*}
	where
	\begin{equation}\label{taua}
		\tau_a = \inf\{ t \geq 0: X_t \leq a \}.
	\end{equation}
	Then we choose $a$ in such a way to satisfy all the conditions of the verification Lemma \ref{HJB}
	are satisfied for $\hat{V}(x, \overline{x})=V_{a}(x, \overline{x})$.
	
	To realize the first goal, we introduce the so-called scale functions.
	We define a Laplace exponent of process $X_t$ as
	\begin{equation*}
		\Psi(\theta) = \frac{1}{t}\log\E e^{\theta X_t}=
		\mu\theta + \frac{\sigma^2\theta^2}{2} - \frac{\lambda\theta}{\theta + \rho}.
	\end{equation*}
	We assumed that $\Q$ is a risk-neutral measure and hence $e^{-rt}S_t$ is a $\Q$-local martingale which is equivalent to
	\begin{equation*}
		\label{Psi1}\Psi(1) = r
	\end{equation*}
	or that
	\begin{equation*}
		\mu = r - \frac{\sigma^2}{2} +  \frac{\lambda}{1+\rho}.
	\end{equation*}
	For $r \geq 0$ the so-called scale function is defined as a continuous function $W^{(r)}:[ 0,\infty ) \rightarrow [ 0,\infty ) $ such that:
	\begin{equation*}
		\int\limits_0^\infty e^{-\beta x}W^{(r)}(x)\dx x = \frac{1}{\Psi(\beta)-r}
	\end{equation*}
	which gives
	\begin{equation}\label{W_sum}
		W^{(r)}(x) = \sum_{i=1}^3 C_i e^{\gamma_i x};
	\end{equation}
	see for details \cite{ZP&PS}.
	The exponents $\gamma_i$ have the following form:
	\begin{align}
		\gamma_1 = 1, \qquad
		%\label{eta1} \\
		\gamma_{2/3} &= \frac{-1}{2(\rho\sigma^2 + \sigma^2)}\left(2\lambda + 2r + \rho^2\sigma^2 + \rho\sigma^2 + 2r\rho \pm 2\sqrt{\omega}\right), \label{eta2}
		%    \gamma_3 &= \frac{-1}{2(\rho\sigma^2 + \sigma^2)}\left(2\lambda + 2r + \rho^2\sigma^2 + \rho\sigma^2 + 2r\rho + 2\sqrt{\omega}\right). \label{eta3}
	\end{align}
	where
	\begin{equation*}
		\omega = \lambda^2 + \lambda (\rho+1)(2r+\rho\sigma^2) + (\rho+1)^2 \left(r - \frac{1}{2}\rho\sigma^2\right)^2.
	\end{equation*}
	Additionally, coefficients $C_i$ are given by:
	\begin{align*}
		C_1 = \frac{2(\gamma_1+\rho)}{\sigma^2(\gamma_1 - \gamma_2)(\gamma_1 - \gamma_3)},\quad %\label{C1}\\
		C_2 = \frac{2(\gamma_2+\rho)}{\sigma^2(\gamma_2 - \gamma_1)(\gamma_2 - \gamma_3)},
		\\
		C_3 = \frac{2(\gamma_3+\rho)}{\sigma^2(\gamma_3 - \gamma_1)(\gamma_3 - \gamma_2)}.
	\end{align*}
	With the first scale function we associate the second one given by
	\begin{equation}\label{Z_def}
		Z^{(r)}(x) = 1 + r\int\limits_{0}^{x} W^{(r)}(y)\dx y=\sum_{i=1}^3 \frac{rC_i}{\gamma_i} e^{\gamma_i x}.
	\end{equation}
	
	The value function $V_a$ for the stopping rule \eqref{taua} is given in the next theorem.
	\begin{Thm}\label{valuea}
		The following holds.
		
		For $x > a$ and $\overline{x} - x < c$:
		\begin{itemize}
			\item[(i)]
			If $\overline{x} < a+c$ then we have
			\[V_a(x,\overline{x})= V_1(x,\overline{x}) + V_2(x,\overline{x}) ( V_3(\overline{x}) +  V_4(\overline{x})(  V_5 + V_6 V_7  )   )\]
			where  $V_1$, $V_2$, $V_3$, $V_4$, $V_5$, $V_6$, $V_7$ are given in (\ref{V1}), (\ref{V2}), (\ref{V3}), (\ref{V4}), (\ref{V5}), (\ref{V6}), (\ref{V7}), respectively.
			\item[(ii)] If $a+c < \overline{x} < \log(K)+c$ then we have
			\[V_a(x,\overline{x})= V_{10}(x,\overline{x}) + V_{11}(x,\overline{x}) ( V_{12}(\overline{x}) +  V_{13}(\overline{x})V_7   )\]
			where  $V_{10}$, $V_{11}$, $V_{12}$, $V_{13}$ are given in (\ref{V10}), (\ref{V11}), (\ref{V12}), (\ref{V13}), respectively.
			\item[(iii)] If $\overline{x} > \log(K)+c$ then we have
			\[V_a(x,\overline{x})= V_{14}(x,\overline{x}) + V_{15}(x,\overline{x}) V_{16}(\overline{x})\]
			where  $V_{14}$, $V_{15}$, $V_{16}$ are given in (\ref{V14}), (\ref{V15}), (\ref{V16}), respectively.
		\end{itemize}
		For $x \leq a$ or $\overline{x} - x \geq c$:
		\begin{itemize}
			\item[(iv)] \[V_a(x,\overline{x})= \left( K - e^x\right)^+. \]
		\end{itemize}
	\end{Thm}
	\begin{proof}
		Assume first that $\overline{x} < a+c$. Then
		\begin{align*}
			V_a&(x,\overline{x}) = \Eq \left[ e^{-r\tau\wedge\tau_D} \left( K - e^{X_{\tau\wedge\tau_D}} \right)^+ \right] =  \Eq\left[ e^{-r\tau} \left( K - e^{X_\tau} \right) \I
			\{ \tau < \tau_{\overline{x}}^+ \} \right]  \\
			\nonumber & +  \Eq\left[ e^{-r\tau_{\overline{x}}^+} \I
			\{ \tau > \tau_{\overline{x}}^+ \} \right]\left( \E^\Q_{\overline{x}}\left[ e^{-r\tau} \left( K - e^{X_\tau} \right) \I
			\{ \tau < \tau_{a+c}^+ \} \right] + \E^\Q_{\overline{x}}\left[ e^{-r\tau_{a+c}^+} \I
			\{ \tau > \tau_{a+c}^+ \} \right]  \right. \\
			\nonumber &  \times \left(
			\E^\Q_{a+c}\left[ e^{-r\tau_D} \left( K - e^{X_{\tau_D}} \right) \I
			\{ \tau_D < \tau_{\log(K)+c}^+ \} \right] +
			\E^\Q_{a+c}\left[ e^{-r\tau_{\log(K)+c}^+}  \I
			\{ \tau_D > \tau_{\log(K)+c}^+ \} \right] \right. \\
			\nonumber & \times \left. \left.
			\E^\Q_{\log(K)+c}\left[ e^{-r\tau_D} \left( K - e^{X_{\tau_D}} \right)^+ \right]
			\right) \right) = V_1(x,\overline{x}) + V_2(x,\overline{x}) ( V_3(\overline{x}) +  V_4(\overline{x})(  V_5 + V_6 V_7  )   ).
		\end{align*}	
		%	\begin{equation}\label{V1}
			%	V_1(x, \overline{x}) = \E_{x, \overline{x}}\left[e^{-r\tau}\left(K - e^{a}\right)\I\{ \tau^-_{a} < \tau^+_{\overline{x}} \} \right] = \left(K - e^{a}\right)\left( Z^{(r)}(x-a) - Z^{(r)}(\overline{x}-a)\frac{W^{(r)}(x-a)}{W^{(r)}(\overline{x}-a)} \right)
			%\end{equation}
			%and
			%\begin{equation}\label{V3}
			%	V_3(\overline{x}) = \E_{\overline{x}, \overline{x}}\left[e^{-r\tau^-_{a}}\left(K - e^{a}\right) \I\{ \tau^-_{a} < \tau^+_{a+c} \} \right] = \left(K - e^{a}\right)\left( Z^{(r)}(\overline{x}-a) - Z^{(r)}(c)\frac{W^{(r)}(\overline{x}-a)}{W^{(r)}(c)} \right).
			%\end{equation}
			From \cite[eq. (2.3)]{MP} we have that
			\begin{equation}\label{V2}
				V_2(x, \overline{x}) = \Eq \left[ e^{-r\tau^+_{\overline{x}}} \I \{ \tau^+_{\overline{x}} < \tau \} \right] = \frac{W^{(r)}(x-a)}{W^{(r)}(\overline{x}-a)}
			\end{equation}
			and
			\begin{equation}\label{V4}
				V_4(\overline{x}) = \E^\Q_{\overline{x}} \left[ e^{-r\tau^+_{a+c}} \I\{ \tau^+_{a+c} < \tau \} \right] = \frac{W^{(r)}(\overline{x}-a)}{W^{(r)}(c)}.
			\end{equation}
			
			From \cite{kyprianou2} we know that
			\begin{equation*}
				e^{X_t - \Psi(1)t}
			\end{equation*}
			is a martingale. To identify other terms, we introduce the following change of measure:
			\begin{equation*}
				\frac{\dx \mathbb{P}}{\dx \Q} \Big{|} _{\F_t} = \frac{e^{X_t - \Psi(1)t}}{e^x}.
			\end{equation*}
			Moreover, we know that $(X, \mathbb{P})$ is also a spectrally negative L\'evy process and its Laplace exponent is given by:
			\begin{equation}\label{psi_P}
				\Psi^\mathbb{P}(\theta) = \Psi(\theta+1) - \Psi(1)=\frac{\sigma^2}{2}\theta^2 + \theta(\mu + \sigma^2) - \frac{\lambda\rho\theta}{(\theta+1+\rho)(\rho+1)} = \frac{\sigma^2}{2}\theta^2 + \theta(\mu + \sigma^2) - \frac{\lambda\rho}{\rho+1}\frac{\theta}{\theta+(\rho+1)}.
			\end{equation}
			Therefore, under $(X,\P)$ the process $X$ has the same form \eqref{X_def} but with new parameters:
			\begin{align*}
				\tilde{\mu} = \mu + \sigma^2,\quad
				\tilde{\lambda} = \frac{\lambda\rho}{\rho+1},\quad
				\tilde{\rho} = \rho +1.
			\end{align*}
			%Let us take $h = 1$ and define the measure $\mathbb{P}$ as in (\ref{radon}).
			Let us now define $W^{(0)}$ scale function for $(X, \mathbb{P})$. We will denote it by $W^\mathbb{P}$. That is,
			\begin{equation*}
				\int\limits_0^\infty e^{-\beta x}W^{\mathbb{P}}(x)\dx x = \frac{1}{\Psi^{\mathbb{P}}(\beta)} \quad \text{for $\beta>0$}.
			\end{equation*}
			By (\ref{psi_P}) observe that
			\begin{equation*}
				\int\limits_0^\infty e^{-\beta x}W^{\mathbb{P}}(x)\dx x = \frac{1}{\Psi(\beta + 1) - \Psi(1)} =  \frac{1}{\Psi(\beta + 1) - r} = \int\limits_0^\infty e^{-x}e^{-\beta x}W^{(r)}(x)\dx x.
			\end{equation*}
			This gives
			\begin{equation*}
				W^{\mathbb{P}}(x) = e^{-x}W^{(r)}(x)
			\end{equation*}
			and that
			\begin{equation}\label{W_P}
				W^{\P}(x) = \sum_{i=1}^3 \tilde{C}_i e^{\tilde{\gamma}_i x},
			\end{equation}
			where $\tilde{C}_i = C_i$ and $\tilde{\gamma}_i = \gamma_i - 1$ for $i=1,2,3$.
			Additionally, we have
			\begin{equation}\label{Z_P}
				Z^{\mathbb{P}}(x) = 1 + 0\int\limits_{0}^{x} W^{\mathbb{P}}(y)\dx y \equiv 1.
			\end{equation}
			%		
			%		From (\ref{WP_Wr}) it is clear that the coefficients of the scale function on $\P$ must satisfy the following equalities:
			%		
			%		Additionally, by looking closer to the Laplace exponent, we get
			%		\begin{align}
				%			\Psi&^\mathbb{P}(\theta) = \Psi(\theta+1) - \Psi(1) = \mu (\theta+1) + \frac{\sigma^2}{2}(\theta+1)^2 - \frac{\lambda(\theta+1)}{\theta+1+\rho} - r\\
				%			%\nonumber &= \frac{\sigma^2}{2}\theta^2 + \theta(\mu + \sigma^2) - \frac{\lambda(\theta+1)-(\theta+1+\rho)\left( \frac{\sigma^2}{2} + \mu -r \right)}{\theta+1+\rho}
				%			\nonumber &= \frac{\sigma^2}{2}\theta^2 + \theta(\mu + \sigma^2) - \frac{\lambda(\theta+1)}{\theta+1+\rho} + \frac{\sigma^2}{2} + \mu -r = \frac{\sigma^2}{2}\theta^2 + \theta(\mu + \sigma^2) - \frac{\lambda(\theta+1)}{\theta+1+\rho} + \frac{\lambda}{1+\rho} \\
				%			\nonumber &=  \frac{\sigma^2}{2}\theta^2 + \theta(\mu + \sigma^2) - \frac{\lambda\rho\theta}{(\theta+1+\rho)(\rho+1)} = \frac{\sigma^2}{2}\theta^2 + \theta(\mu + \sigma^2) - \frac{\lambda\rho}{\rho+1}\frac{\theta}{\theta+(\rho+1)}.
				%		\end{align}
			%		Therefore, for $(X,\P)$ we have:
			%		\begin{align}
				%			\tilde{\mu} = \mu + \sigma^2 \\
				%			\tilde{\lambda} = \frac{\lambda\rho}{\rho+1} \\
				%			\tilde{\rho} = \rho +1.
				%		\end{align}
			%	
			%Scale functions $W$ and $Z$ defined on $\P$ allow us to tackle $V_1$ and $V_2$.
			Now sing \cite[eq. (2.4)]{MP}, we get
			\begin{align}
				\nonumber	V_1(x, \overline{x}) &= \Eq\left[ e^{-r\tau} \left( K - e^{X_\tau} \right) \I
				\{ \tau < \tau_{\overline{x}}^+ \} \right]  = K\Eq \left[ e^{-r\tau} \I
				\{ \tau < \tau_{\overline{x}}^+ \} \right] - \Eq\left[ e^{-r\tau + X_\tau} \I
				\{ \tau < \tau_{\overline{x}}^+ \} \right] \\
				\nonumber & =  K\Eq \left[ e^{-r\tau} \I
				\{ \tau < \tau_{\overline{x}}^+ \} \right] - e^x \E^\P_{x, \overline{x}}\left[ \I
				\{ \tau < \tau_{\overline{x}}^+ \} \right] \\
				\nonumber & =  K\left( Z^{(r)}(x-a) - Z^{(r)}(\overline{x}-a)\frac{W^{(r)}(x-a)}{W^{(r)}(\overline{x}-a)} \right) -e^x\left( Z^{\P}(x-a) - Z^{\P}(\overline{x}-a)\frac{W^{\P}(x-a)}{W^{\P}(\overline{x}-a)} \right) \\
				\nonumber & = K\left( Z^{(r)}(x-a) - Z^{(r)}(\overline{x}-a)\frac{W^{(r)}(x-a)}{W^{(r)}(\overline{x}-a)} \right) - \left( e^x - \frac{e^x W^\P(x-a)}{W^\P(\overline{x}-a)} \right) \\
				& = K\left( Z^{(r)}(x-a) - Z^{(r)}(\overline{x}-a)\frac{W^{(r)}(x-a)}{W^{(r)}(\overline{x}-a)} \right) - \left( e^x - \frac{e^{\overline{x}} W^{(r)}(x-a)}{W^{(r)}(\overline{x}-a)} \right)\label{V1}
			\end{align}
			and
			\begin{align}
				\nonumber	V_3(\overline{x}) &= \E^\Q_{\overline{x}}\left[ e^{-r\tau} \left( K - e^{X_\tau} \right) \I
				\{ \tau < \tau_{a+c}^+ \} \right]  = K\E^\Q_{\overline{x}} \left[ e^{-r\tau} \I
				\{ \tau < \tau_{a+c}^+ \} \right] - \E^\Q_{\overline{x}}\left[ e^{-r\tau + X_\tau} \I
				\{ \tau < \tau_{a+c}^+ \} \right] \\
				\nonumber & =  K\E^\Q_{\overline{x}} \left[ e^{-r\tau} \I
				\{ \tau < \tau_{a+c}^+ \} \right] - e^x \E^\P_{\overline{x}}\left[ \I
				\{ \tau < \tau_{a+c}^+ \} \right] \\
				\nonumber &=  K\left( Z^{(r)}(\overline{x}-a) - \frac{Z^{(r)}(c)}{W^{(r)}(c)}W^{(r)}(\overline{x}-a) \right) -e^{\overline{x}}\left( Z^{\P}(\overline{x}-a) - \frac{Z^{\P}(c)}{W^{\P}(c)}W^{\P}(\overline{x}-a) \right) \\
				\nonumber & =K\left( Z^{(r)}(\overline{x}-a) - \frac{Z^{(r)}(c)}{W^{(r)}(c)}W^{(r)}(\overline{x}-a) \right) - \left( e^{\overline{x}} - \frac{e{\overline{x}} W^\P(\overline{x}-a)}{W^\P(c)} \right) \\
				& =K\left( Z^{(r)}(\overline{x}-a) - \frac{Z^{(r)}(c)}{W^{(r)}(c)}W^{(r)}(\overline{x}-a) \right) - \left( e^{\overline{x}} - \frac{e^{a+c} W^{(r)}(\overline{x}-a)}{W^{(r)}(c)} \right).\label{V3}
			\end{align}
			
			Moving on to $V_5$, observe that	
			\begin{align*}
				V_5 &= \E_{a+c} \left[ e^{-r\tau_D}\left( K - e^{X_{\tau_D}} \right) \I \{ \tau_D < \tau_{\log(K)+c}^+ \} \right] \\
				\nonumber & = K\E_{a+c} \left[ e^{-r\tau_D} \I \{ \tau_D < \tau_{\log(K)+c}^+ \} \right] - e^{a+c}\E^\P_{a+c} \left[ e^{-r\tau_D} \I \{ \tau_D < \tau_{\log(K)+c}^+ \} \right].
			\end{align*}

			Let us introduce the following notations to deal with the last terms
			\begin{align}
				&\eta^\Q = \frac{W^{(r)'}(c)}{W^{(r)}(c)}, \label{def_eta} \\
				\nonumber &F^\Q(y) = \eta^\Q e^{-y\eta^\Q},\quad y \in \mathbb{R}_+,  \\
				&\Delta^\Q = \frac{\sigma^2}{2}\left[ W^{(r)'}(c) - \frac{1}{\eta^{\Q}}W^{(r)''}(c) \right], \label{def_Delta} \\
				\nonumber &R(r,\dx y) = \left[\frac{1}{\eta^{\Q}} W^{(r)'}(y)\dx y - W^{(r)}(y)\dx y \right].
			\end{align}
			We will also need
			\begin{align*}
				\overline{X}_t =  \sup_{0 \leq u \leq t} X_u \ \vee\ \overline{x},  \quad
				\underline{X}_t =  \inf_{0 \leq u \leq t} X_u, \quad
				D_t = \overline{X}_t - X_t.
			\end{align*}
			By
			\begin{equation*}
				\Lambda(y-c-\dx h) = \lambda\rho e^{\rho(y-c-h)}\dx h, \ h \in (0, \infty),
			\end{equation*}
			we denote the L\'evy measure of the L\'evy process $X_t$.		
			Now, let us define the following two events:
			\begin{align*}
				A_o &= \{ \underline{X}_{\tau_D} \geq u;\ \overline{X}_{\tau_D} \in \dx v; D_{\tau_D-} \in \dx y; D_{\tau_D} - c \in \dx h \}, \\
				A_c &= \{ \underline{X}_{\tau_D} \geq u;\ \overline{X}_{\tau_D} \in \dx v; D_{\tau_D-} = c \}.
			\end{align*}
			The first one is associated with drawdown exceeding the threshold with a Poissonian jump, the latter is related to the hitting the threshold by creeping. From \cite[eq. (3.10, 3.11)]{MP} we have
			\begin{equation*}
				\E_x \left[ e^{-r\tau_D} \I_{A_o}\right] = \frac{W^{(r)} ((x-u)\wedge c)}{W^{(r)}(c)}F^\Q(v - (x\vee (u+c)))\dx v R(r,\dx y)	\Lambda(y-c-\dx h)
			\end{equation*}
and
			\begin{equation*}
				\E_x \left[ e^{-r\tau_D} \I_{A_c}\right] = \frac{W^{(r)} ((x-u)\wedge c)}{W^{(r)}(c)}F^\Q(v - (x\vee (u+c)))\Delta^\Q.
			\end{equation*}		
			%To match the expected values in (\ref{V5}), let us define events
			We can represent events $A_o$ and $A_c$ as follows
			\begin{align*}
				A_o^5 = \{& \overline{X}_{\tau_D} \in \dx v,\ v \in [a+c, \log(K) +c); D_{\tau_D-} \in \dx y,\ y \in \left[ 0,c \right); \\ \nonumber & D_{\tau_D} - c \in \dx h,\ h \in (0, \infty)
				%,\ h \sim \text{Exp}(\rho)
				\},
			\end{align*}
			\begin{equation*}
				A_c^5 = \{ \underline{X}_{\tau_D} \geq a;\ \overline{X}_{\tau_D} \in \dx v,\ v \in [a+c, \log(K) +c);\  D_{\tau_D} =c \}.
			\end{equation*}
			For the first event, $A_o^5$, the first condition $\underline{X}_{\tau_D} \geq u$ disappears as the jump sizes are unbounded, allowing $\underline{X}_{\tau_D}$ to be arbitrarily small. Additionally, observe that no matter if we take $u=-\infty$ or $u=a$, we get $(a+c-u)\wedge c = c$ and $a+c \vee (u+c) = a+c$. As a result, we get $\frac{W^{(r)} ((x-u)\wedge a)}{W^{(r)}(a)} = 1$.
			This gives
			\begin{equation*}
				\E_{a+c} \left[ e^{-r\tau_D} \I \{ \tau_D < \tau_{\log(K)+c}^+ \} \right] = \int\limits_{a+c}^{\log(K)+c}F^\Q(v-(a+c))\dx v \left[ \Delta^\Q + \int\limits_{0}^{\infty}\int\limits_{0}^{c} R(r,\dx y) 	\Lambda(y-c-\dx h)  \right].
			\end{equation*}
			The first integral equals
			\begin{align*}
				\int\limits_{a+c}^{\log(K)+c}&F^\Q(v-(a+c))\dx v = \int\limits_{a+c}^{\log(K)+c} \eta^\Q e^{-(v-(a+c))\eta^\Q} \dx v \\
				\nonumber &= e^{(a+c)\eta^\Q}\left(  e^{-(a+c)\eta^\Q} - e^{-(\log(K)+c)\eta^\Q}\right) = 1 - e^{(a-\log(K))\eta^\Q}
			\end{align*}
			and the double integral from the square bracket is
			\begin{align}
				\nonumber \int\limits_{0}^{\infty}\int\limits_{0}^{c} &R(r,\dx y) 	\Lambda(y-c-\dx h) = \int\limits_{0}^{c}\int\limits_{0}^{\infty} \left[\eta^{\Q-1} W^{(r)'}(y) - W^{(r)}(y) \right] \lambda\rho e^{\rho(y-d-h)} \dx h \dx y \\
				\nonumber &= \lambda e^{-\rho c} \int\limits_{0}^{\infty} \rho e^{-\rho h} \dx h \int\limits_{0}^{c} e^{\rho y} \left[\eta^{\Q-1} W^{(r)'}(y) - W^{(r)}(y) \right] \dx y \\
				\nonumber &= \lambda e^{-\rho c} \int\limits_{0}^{\infty} \rho e^{-\rho h} \dx h \int\limits_{0}^{c} e^{\rho y} \left[\eta^{\Q-1} \sum_{i=1}^3 C_i \gamma_i e^{\gamma_i x} -\sum_{i=1}^3 C_i e^{\gamma_i x} \right] \dx y \\
				&= \lambda e^{-\rho c} \sum_{i=1}^3 \frac{C_i}{\gamma_i + \rho} \left( \frac{\gamma_i}{\eta^\Q}-1\right)\left(e^{c(\gamma_i+\rho)}-1\right) := \Gamma_\Q. \label{deriveGamma}
			\end{align}
			%Here, we introduce $\Gamma_\Q$ for better readability.
			Observe that $\sum_{i=1}^3 \frac{C_i}{\gamma_i+\rho}=0$ since $C_i = \frac{2}{\sigma^2} \frac{\gamma_i + \rho}{(\gamma_i-\gamma_j)(\gamma_i-\gamma_k)}$ for $i,j,k \in \{1,2,3\},\ i\neq j, i\neq k, j\neq k$.
			Additionally,
			\begin{equation}\label{sum_Ci_0}
				\lambda e^{-\rho c} \sum_{i=1}^{3} \frac{C_i}{\gamma_i+\rho} \frac{\gamma_i}{\eta^\Q} = \lambda e^{-\rho c} \frac{1}{\eta^\Q}\sum_{i=1}^{3} \frac{C_i \gamma_i}{\gamma_i + \rho} = 0
			\end{equation}
			because
			\begin{align*}
				\sum_{i=1}^{3} &\frac{C_i \gamma_i}{\gamma_i + \rho} = \frac{2}{\sigma^2}\sum_{i=1}^{3} \frac{\gamma_i + \rho}{(\gamma_i-\gamma_j)(\gamma_i-\gamma_k)} \frac{\gamma_i}{\gamma_i + \rho} \\
				\nonumber	&= \frac{2}{\sigma^2} \left[ \frac{\gamma_1}{(\gamma_1-\gamma_2)(\gamma_1-\gamma_3)} + \frac{\gamma_2}{(\gamma_2-\gamma_1)(\gamma_2-\gamma_3)} + \frac{\gamma_3}{(\gamma_3-\gamma_1)(\gamma_3-\gamma_2)} \right] \\
				\nonumber	&= \frac{2}{\sigma^2(\gamma_1-\gamma_2)(\gamma_1-\gamma_3)(\gamma_2-\gamma_3)} [ \gamma_1(\gamma_2-\gamma_3) - \gamma_2(\gamma_1-\gamma_3) + \gamma_3(\gamma_1-\gamma_2) ] = 0.
			\end{align*}
			Therefore, we can simplify $\Gamma_\Q$ in the following way
			\begin{equation}\label{def_Gamma}
				\Gamma_\Q =	\lambda  \sum_{i=1}^3 \frac{C_i}{\gamma_i + \rho} \left( \frac{\gamma_i}{\eta^\Q} - 1\right)e^{\gamma_i c}.
			\end{equation}
			To handle the expected value on $\P$ measure, we need to introduce the similar notations as for $\Q$ measure:
			\begin{equation*}
				\eta^\P = \frac{W^{\P'}(c)}{W^{\P}(c)} = \frac{e^{-c}\left( W^{(r)'}(c) - W^{(r)}(c) \right)}{e^{-c}W^{(r)}(c)} = \eta^\Q - 1,
			\end{equation*}
			\begin{equation*}
				F^\P(y) = \eta^\P e^{-y\eta^\P},\quad y \in \mathbb{R}_+
			\end{equation*}
			and
			\begin{align*}
				\Delta&^\P = \frac{\sigma^2}{2}\left[ W^{\P'}(c) - \eta^{\P-1}W^{\P''}(c) \right] \\
				\nonumber &= \frac{\sigma^2}{2} e^{-c} \left( W^{(r)'}(c) - W^{(r)}(c) - \frac{W^{(r)}(c)\left( W^{(r)''}(c) - 2W^{(r)'}(c) + W^{(r)}(c) \right)}{W^{(r)'}(c) - W^{(r)}(c)} \right) \\
				\nonumber &= \frac{\sigma^2}{2} e^{-c} \left( W^{(r)'}(c) -  \frac{W^{(r)}(c)\left( W^{(r)''}(c) -W^{(r)'}(c)  \right)}{W^{(r)'}(c) - W^{(r)}(c)} \right) \\
				\nonumber &= \frac{\sigma^2}{2} e^{-c} \left( W^{(r)'}(c) -  \frac{ W^{(r)''}(c) -W^{(r)'}(c)  }{\eta^\Q-1} \right) = \frac{\eta^\Q}{\eta^\Q-1} \frac{\sigma^2}{2} e^{-c} \left( W^{(r)'}(c) - \frac{W^{(r)''}(c)}{\eta^\Q} \right) \\
				\nonumber &=  \frac{\eta^\Q}{\eta^\Q-1} e^{-c} \Delta^\Q
			\end{align*}
			Finally, let
			\begin{align*}
				\Gamma_\P &= \tilde{\lambda} e^{-\tilde{\rho} c} \sum_{i=1}^3 \frac{\tilde{C}_i}{\tilde{\gamma}_i + \tilde{\rho}} \left( \frac{\tilde{\gamma}_i}{\eta^\P}-1\right)\left(e^{c(\tilde{\gamma}_i+\tilde{\rho})}-1\right) \\
				\nonumber &= \lambda \frac{\rho}{\rho+1} e^{-c} e^{-\rho c} \sum_{i=1}^3 \frac{C_i}{\gamma_i + \rho} \left( \frac{\gamma_i-1}{\eta^\Q-1}-1\right)\left(e^{c(\gamma_i+\rho)}-1\right) \\
				\nonumber &=  \frac{\rho e^{-c}}{\rho+1} \frac{\eta^\Q}{\eta^\Q-1}  \lambda e^{-\rho c} \sum_{i=1}^3 \frac{C_i}{\gamma_i + \rho} \left( \frac{\gamma_i-1}{\eta^\Q}-\frac{\eta^\Q-1}{\eta^\Q}\right)\left(e^{c(\gamma_i+\rho)}-1\right) \\
				\nonumber &=  \frac{\rho e^{-c}}{\rho+1} \frac{\eta^\Q}{\eta^\Q-1}  \lambda e^{-\rho c} \sum_{i=1}^3 \frac{C_i}{\gamma_i + \rho} \left( \frac{\gamma_i-}{\eta^\Q}-1\right)\left(e^{c(\gamma_i+\rho)}-1\right) = \frac{\rho e^{-c}}{\rho+1} \frac{\eta^\Q}{\eta^\Q-1} \Gamma_\Q.
			\end{align*}

			By combining everything together, we get
			\begin{align}
				\nonumber V_5 &= K\E_{a+c} \left[ e^{-r\tau_D} \I \{ \tau_D < \tau_{\log(K)+c}^+ \} \right] - e^{a+c}\E^\P_{a+c} \left[ e^{-r\tau_D} \I \{ \tau_D < \tau_{\log(K)+c}^+ \} \right] \\
				\nonumber & = K\left[ \left(1 - e^{(a-\log(K))\eta^\Q}\right)\left( \Delta^\Q + \Gamma_\Q \right)  \right] - e^{a+c}\left[ \left(1 - e^{(a-\log(K))\eta^\P}\right)\left( \Delta^\P + \Gamma_\P \right)  \right] \\
				\nonumber & = K\left[ \left(1 - e^{(a-\log(K))\eta^\Q}\right)\left( \Delta^\Q + \Gamma_\Q \right)  \right] - e^{a+c}\left[ \left(1 - e^{(a-\log(K))\eta^\Q + \log(K)-a}\right)\left( \Delta^\P + \Gamma_\P \right)  \right] \\
				\nonumber & = \left[ \left(K - Ke^{(a-\log(K))\eta^\Q}\right)\left( \Delta^\Q + \Gamma_\Q \right)  \right] - \left[ \left(e^{a+c} - Ke^ce^{(a-\log(K))\eta^\Q}\right)\left( \Delta^\P + \Gamma_\P \right)  \right] \\
				\nonumber & = \left[ \left(K - Ke^{(a-\log(K))\eta^\Q}\right)\left( \Delta^\Q + \Gamma_\Q \right)  \right] - \left[ \left(e^{a} - Ke^{(a-\log(K))\eta^\Q}\right)\frac{\eta^\Q}{\eta^\Q-1}\left( \Delta^\Q + \frac{\rho}{\rho+1}\Gamma_\Q \right)  \right] \\
				& = Ke^{(a-\log(K))\eta^\Q}\frac{\Delta^\Q + \Gamma_\Q \frac{\rho+1-\eta^\Q}{\rho+1}}{\eta^\Q-1} - e^a \frac{\eta^\Q}{\eta^\Q-1}\left( \Delta^\Q + \frac{\rho}{\rho+1}\Gamma_\Q \right) + K\left( \Delta^\Q + \Gamma_\Q \right). \label{V5}
			\end{align}		
			Now, let us consider an event, when the first drawdown of the process $X_t$ starting from $a+c$ occurs after the process hits level $\log(K) + c$. We can then split the time until drawdown into two sub-intervals: from 0 to $\tau_{\log(K)+c}^+$ and from $\tau_{\log(K)+c}^+$ to $\tau_D$. As time before reaching $\log(K)+c$ is independent from time to first drawdown starting from $\log(K)+c$, we can write the following relation:
			\begin{equation}\label{V6_1st}
				V_6 = \E_{a+c}\left[ e^{-r\tau_{\log(K)+c}^+}  \I
				\{ \tau_D > \tau_{\log(K)+c}^+ \} \right] = \frac{\E_{a+c}\left[ e^{-r\tau_D}  \I
					\{ \sup_{t \in (\tau^+_{a+c},\tau_D)} X_t \geq \log(K)+c \} \right]}{\E_{\log(K)+c}\left[ e^{-r\tau_D}\right]} := \frac{V_9}{V_8}.
			\end{equation}	
			Let us first handle the denominator. Again, we want to specify events $A_o$ and $A_c$ for $V_8$:		
			\begin{align*}
				A_o^8 = \{& \overline{X}_{\tau_D} \in \dx v,\ v \in [\log(K)+c, \infty); D_{\tau_D-} \in \dx y,\ y \in \left[ 0,c \right); \\ \nonumber & D_{\tau_D} - c \in \dx h,\ h \in (0, \infty)
				%\ h \sim \text{Exp}(\rho)
				\}
			\end{align*}
and
			\begin{equation*}
				A_c^8 = \{ \underline{X}_{\tau_D} \geq \log(K);\ \overline{X}_{\tau_D} \in \dx v,\ v \in [\log(K)+c,\infty);\  D_{\tau_D} =c \}.
			\end{equation*}
			With these, we get the formula for $V_8$ as follows
			\begin{equation*}
				V_8 = \int\limits_{\log(K)+c}^{\infty}F^\Q(v-(\log(K)+c))\dx v \left[ \Delta^\Q + \int\limits_{0}^{\infty}\int\limits_{0}^{c} R(r,\dx y) 	\Lambda(y-c-\dx h)  \right] = \Delta^\Q + \Gamma_Q
			\end{equation*}	
			as the first integral is equal to $1$ and the double integral is the same as one derived for in $V_5$ in (\ref{deriveGamma}). On the other hand, from formula (3.3) from \cite{MP} we have:
			\begin{equation*}
				V_8 = Z^{(r)}(c) - r\frac{W^{(r)}(c)^2}{W^{(r)'}(c)}
			\end{equation*}
			and therefore
			\begin{equation}\label{delta_gamma}
				\Delta^\Q + \Gamma_\Q = Z^{(r)}(c) - r\frac{W^{(r)}(c)^2}{W^{(r)'}(c)}.
			\end{equation}		
			Similarly, using the same argument for $\P$ counterparts and applying equality (\ref{Z_P}), we get
			\begin{equation*}
				\Delta^\P + \Gamma_\P = Z^{\P}(c) - 0\cdot\frac{W^{\P}(c)^2}{W^{\P'}(c)} = 1.
			\end{equation*}
			Now, observe that for $V_9$, event $A_o$ is the same as for $V_8$: 
			\begin{align*}
				A_o^9 =  A_o^8 = \{& \overline{X}_{\tau_D} \in \dx v,\ v \in [\log(K)+c, \infty); D_{\tau_D-} \in \dx y,\ y \in \left[ 0,c \right); \\ \nonumber & D_{\tau_D} - c \in \dx h,\ h \in (0, \infty)
				%,\ h \sim \text{Exp}(\rho)
				\}
			\end{align*}
and only the lower bound of $\underline{X}_{\tau_D}$ changes for $A_c$ in the following way
			\begin{equation*}
				A_c^9 = \{ \underline{X}_{\tau_D} \geq a;\ \overline{X}_{\tau_D} \in \dx v,\ v \in [\log(K)+c,\infty);\  D_{\tau_D} =c \}.
			\end{equation*}	
			
			Similarly to $V_8$, we get
			\begin{equation*}
				V_9 = \int\limits_{\log(K)+c}^{\infty}F^\Q(v-(a+c))\dx v \left[ \Delta^\Q + \int\limits_{0}^{\infty}\int\limits_{0}^{c} R(r,\dx y) 	\Lambda(y-c-\dx h)  \right] = e^{\eta^\Q(a-\log(K))}(\Delta^\Q + \Gamma_Q).
			\end{equation*}	
			Finally, by (\ref{V6_1st})
			\begin{equation}
				V_6 = \frac{V_9}{V_8} = e^{\eta^\Q(a-\log(K))}.\label{V6}
			\end{equation}
			
			For the last component of $V$, i.e. $V_7$, we only need the event $A_o$. It is impossible to get a non-zero payout from the option for the stock price starting from $Ke^c$, if the drawdown does not occur by a Poissonian jump. Hence in this case
			\begin{align}\label{A7}
				A_o^7 = \{& \overline{X}_{\tau_D} \in \dx v,\ v \in (\log(K)+c, \infty); D_{\tau_D-} \in \dx y,\ y \in \left[ 0,c \right); \\ \nonumber & D_{\tau_D} - c \in \dx h,\ h \in (v-c-\log(K), \infty)
				%\ h \sim \text{Exp}(\rho)
				\}.
			\end{align}
			and then
			\begin{equation*}
				V_7 = \E_{\log(K)+c}\left[ e^{-r\tau_D} \left( K - e^{X_{\tau_D}} \right)^+ \right] = K \E_{\log(K)+c}\left[ e^{-r\tau_D} \I \{ A_o^7 \} \right] - Ke^c \E^\P_{\log(K)+c}\left[ \I \{ A_o^7 \} \right].
			\end{equation*}
			Furthermore, note that		
			\begin{align*}
				\E&_{\log(K)+c}\left[ e^{-r\tau_D} \I \{ A_o^7 \} \right] = \int\limits_{\log(K)+c}^{\infty} \int\limits_{v-c-\log(K)}^{\infty}\int\limits_{0}^{c}F^\Q(v-(\log(K)+c))   R(r,\dx y) 	\Lambda(y-c-\dx h)   \dx v
				\\
				\nonumber&= \int\limits_{\log(K)+c}^{\infty} \int\limits_{v-c-\log(K)}^{\infty}\int\limits_{0}^{c} \eta^\Q e^{-( v - (\log(K)+c) )\eta^\Q} \left[ \frac{W^{(r)'}(y)}{\eta^{\Q}} - W^{(r)}(y) \right] \lambda\rho e^{\rho(y-c-h)} \dx y \dx h \dx v
				\\
				%			\nonumber&= \int\limits_{\log(K)+c}^{\infty} \int\limits_{v-c-\log(K)}^{\infty} \eta^\Q e^{-( v - (\log(K)+c) )\eta^\Q - \rho h}   \dx h \dx v \int\limits_{0}^{c} \lambda e^{\rho(y-c)} \left[ \frac{W^{(r)'}(y)}{\eta^{\Q}} - W^{(r)}(y) \right] \dx y
				%			\\
				\nonumber&=  \int\limits_{\log(K)+c}^{\infty} \eta^\Q  e^{-( v - (\log(K)+c) )\eta^\Q} \int\limits_{v-c-\log(K)}^{\infty} \rho e^{ - \rho h} \dx h \dx v \int\limits_{0}^{c} \lambda e^{\rho(y-c)} \left[ \frac{W^{(r)'}(y)}{\eta^{\Q}} - W^{(r)}(y) \right] \dx y
				\\
				\nonumber&= \Gamma_Q \int\limits_{\log(K)+c}^{\infty} \eta^\Q  e^{-( v - (\log(K)+c) )\eta^\Q} e^{\rho(\log(K)+c-v)} \dx v  = \Gamma_Q e^{\left(\rho + \eta^\Q\right)(\log(K)+c)}\frac{\eta^\Q}{\eta^\Q + \rho} \left[ -e^{-v\left(\eta^\Q + \rho\right)} \right]^{\infty}_{\log(K)+c} \\
				\nonumber&=  \frac{\eta^\Q \Gamma_\Q}{\eta^\Q + \rho} e^{\left(\rho + \eta^\Q\right)(\log(K)+c) - \left(\rho + \eta^\Q\right)(\log(K)+c)} =  \frac{\eta^\Q \Gamma_\Q}{\eta^\Q + \rho}.
			\end{align*}
			Similarly:
			\begin{equation*}
				\E^\P_{\log(K)+c}\left[ e^{-r\tau_D} \I \{ A_o^7 \} \right] = \frac{\eta^\P \Gamma_\P}{\eta^\P + \tilde{\rho}} = \frac{\left(\eta^\Q -1\right) \frac{\rho e^{-c}}{\rho+1} \frac{\eta^\Q}{\eta^\Q-1} \Gamma_\Q}{\eta^\Q - 1+ \rho + 1 } = \frac{\rho}{\rho+1}\frac{\eta^\Q \Gamma_\Q}{\eta^\Q + \rho}e^{-c}.
			\end{equation*}
			Finally, we have
			\begin{equation}
				V_7 = K\left( \frac{\eta^\Q \Gamma_\Q}{\eta^\Q + \rho} - \frac{\rho}{\rho+1}\frac{\eta^\Q \Gamma_\Q}{\eta^\Q + \rho}  \right) = \frac{K\eta^\Q\Gamma_\Q}{\left(\eta^\Q + \rho\right)(\rho+1)}.\label{V7}
			\end{equation}
			This completes the derivation of value function $V$ for $\overline{x} < a+c$.
			
			Let us now consider the case where $a+c \leq \overline{x} < \log(K) + c$. In this case we have
			\begin{align*}
				V_a&(x,\overline{x}) = \Eq \left[ e^{-r\tau\wedge\tau_D} \left( K - e^{X_{\tau\wedge\tau_D}} \right)^+ \right] =  \Eq\left[ e^{-r\tau_{\overline{x}-c}^-} \left( K - e^{X_{\tau_{\overline{x}-c}^-}} \right) \I
				\{ \tau_{\overline{x}-c}^- < \tau_{\overline{x}}^+ \} \right]  \\
				\nonumber & +  \Eq\left[ e^{-r\tau_{\overline{x}}^+} \I
				\{ \tau_{\overline{x}-c}^- > \tau_{\overline{x}}^+ \} \right]\left( \E_{\overline{x}}\left[ e^{-r\tau_D} \left( K - e^{X_{\tau_D}} \right) \I
				\{ \tau_D < \tau_{\log(K)+c}^+ \} \right] \right. \\
				\nonumber & \left. + \E_{\overline{x}}\left[ e^{-r\tau_{\log(K)+c}^+} \I
				\{ \tau_{\log(K+c)^+ < \tau_D} \} \right] \E_{\log(K)+c}\left[ e^{-r\tau_D} \left( K - e^{X_{\tau_D}} \right)^+ \right] \right) = V_{10}(x,\overline{x}) + V_{11}(x,\overline{x}) (V_{12}(\overline{x})  + V_{13}(\overline{x})V_7).
			\end{align*}
			Observe that $V_7$ appears in both cases when $\overline{x}$ is smaller or greater than $a+c$.
			
			Calculations of $V_{10}$ are similar to calculations of $V_3$:
			\begin{align}
				\nonumber V_{10}&(x,\overline{x}) = \Eq\left[ e^{-r\tau_{\overline{x}-c}^-} \left( K - e^{X_{\tau_{\overline{x}-c}^-}} \right) \I
				\{ \tau_{\overline{x}-c}^- < \tau_{\overline{x}}^+ \} \right] = K\E_x\left[ e^{-r\tau_{\overline{x}-c}^-} \I
				\{ \tau_{\overline{x}-c}^- < \tau_{\overline{x}}^+ \} \right] \\
				\nonumber &  - e^x \E^\P_x\left[  \I
				\{ \tau_{\overline{x}-c}^- < \tau_{\overline{x}}^+ \} \right] =  K\left( Z^{(r)}(x+c-\overline{x}) - \frac{Z^{(r)}(c)}{W^{(r)}(c)}W^{(r)}(x+c-\overline{x}) \right) \\
				\nonumber &  -e^{\overline{x}}\left( Z^{\P}(x+c-\overline{x}) - \frac{Z^{\P}(c)}{W^{\P}(c)}W^{\P}(x+c-\overline{x}) \right) \\
				& = K\left( Z^{(r)}(x+c-\overline{x}) - \frac{Z^{(r)}(c)}{W^{(r)}(c)}W^{(r)}(x+c-\overline{x}) \right) - \left( e^x - \frac{e^{\overline{x}}W^{(r)}(x+c-\overline{x})}{W^{(r)}(c)} \right). \label{V10}
			\end{align}
			
			Now, similarly to $V_2$ and $V_4$, we have:
			\begin{equation}\label{V11}
				V_{11}(x, \overline{x}) = \Eq\left[ e^{-r\tau_{\overline{x}}^+} \I
				\{ \tau_{\overline{x}-c}^- > \tau_{\overline{x}}^+ \} \right] = \frac{W^{(r)}(x+c-\overline{x})}{W^{(r)}(c)}.
			\end{equation}
			
			The term $V_{12}$ can be identified in a similar way as it was done for $V_5$. We introduce
			\begin{align*}
				A_o^{12} = \{& \overline{X}_{\tau_D} \in \dx v,\ v \in (\overline{x}, \log(K)+c); D_{\tau_D-} \in \dx y,\ y \in \left[ 0,c \right); \\ \nonumber & D_{\tau_D} - c \in \dx h,\ h \in (0, \infty)
				%\ h \sim \text{Exp}(\rho)
				\}
			\end{align*}
			and
			\begin{equation*}
				A_c^{12} = \{ \underline{X}_{\tau_D} \geq \overline{x}-c;\ \overline{X}_{\tau_D} \in \dx v,\ v \in (\overline{x},\log(K)+c);\  D_{\tau_D} =c \}.
			\end{equation*}
			Then
			\begin{align}
				\nonumber V_{12}&(\overline{x}) = \E_{\overline{x}}\left[ e^{-r\tau_D} \left( K - e^{X_{\tau_D}} \right) \I
				\{ \tau_D < \tau_{\log(K)+c}^+ \} \right] \\
				\nonumber & = K\E_{\overline{x}}\left[ e^{-r\tau_D} \I
				\{ \tau_D < \tau_{\log(K)+c}^+ \} \right] - e^{\overline{x}}\E^\P_{\overline{x}}\left[  \I
				\{ \tau_D < \tau_{\log(K)+c}^+ \} \right] \\
				\nonumber & = K\int\limits^{\log(K)+c}_{\overline{x}}F^\Q(v-\overline{x})\dx v \left( \Delta^\Q + \Gamma_Q \right)
				- e^{\overline{x}} \int\limits^{\log(K)+c}_{\overline{x}}F^\P(v-\overline{x})\dx v \left( \Delta^\P + \Gamma_P \right)
				\\
				\nonumber & = K \left(1 - e^{-(\log(K)+c-\overline{x})\eta^\Q}\right)\left( \Delta^\Q + \Gamma_\Q \right)   - e^{\overline{x}}\left[ \left(1 - e^{-(\log(K)+c-\overline{x})\eta^\P}\right)\left( \Delta^\P + \Gamma_\P \right)  \right]
				\\
				\nonumber & = K \left(1 - e^{-(\log(K)+c-\overline{x})\eta^\Q}\right)\left( \Delta^\Q + \Gamma_\Q \right)   - e^{\overline{x}}\left[ \left(1 - e^{-(\log(K)+c-\overline{x})\eta^\Q} Ke^{c-\overline{x}} \right)\frac{\eta^\Q e^{-c}}{\eta^\Q-1}\left( \Delta^\Q + \frac{\rho}{\rho+1}\Gamma_\Q \right)  \right]
				\\
				\nonumber & = K \left(1 - e^{-(\log(K)+c-\overline{x})\eta^\Q}\right)\left( \Delta^\Q + \Gamma_\Q \right)   - \frac{\eta^\Q}{\eta^\Q-1}\left[ \left(e^{\overline{x}-c} - Ke^{-(\log(K)+c-\overline{x})\eta^\Q}  \right)\left( \Delta^\Q + \frac{\rho}{\rho+1}\Gamma_\Q \right)  \right]
				\\
				& = K \left(1 + \frac{e^{-(\log(K)+c-\overline{x})\eta^\Q}}{\eta^\Q-1}\right)\left( \Delta^\Q + \Gamma_\Q \right)   - \frac{\eta^\Q}{\eta^\Q-1}\left[ e^{\overline{x}-c}\Delta^\Q + \frac{\Gamma_Q}{\rho+1}\left( \rho e^{\overline{x}-c} +  Ke^{-(\log(K)+c-\overline{x})\eta^\Q} \right)  \right]. \label{V12}
			\end{align}
			Following the analysis of $V_6$ changing the starting point of $X_t$ from $\overline{x}$ to $a+c$ we can write
			\begin{equation}\label{V13}
				V_{13} = \E_{\overline{x}}\left[ e^{-r\tau_{\log(K)+c}^+} \I
				\{ \tau_{\log(K+c)^+ < \tau_D} \} \right] = e^{-\eta^\Q (\log(K) + c - \overline{x})}.
			\end{equation}

			Now, let us consider the case when $\overline{x} \geq \log(K) + c$. Then
			\begin{align}\label{V_3}
				V_a&(x,\overline{x}) = \Eq \left[ e^{-r\tau\wedge\tau_D} \left( K - e^{X_{\tau\wedge\tau_D}} \right)^+ \right] =  \Eq\left[ e^{-r\tau_D} \left( K - e^{X_{\tau_D}} \right)^+ \I
				\{ \tau_D < \tau_{\overline{x}}^+ \} \right]
				\\
				\nonumber &+ \Eq\left[ e^{-r\tau_{\overline{x}}^+} \I
				\{\tau_{\overline{x}}^+ < \tau_{\overline{x}+c}^- \} \right]  \E_{\overline{x}}\left[ e^{-r\tau_D} \left( K - e^{X_{\tau_D}} \right)^+ \right] = V_{14}(x,\overline{x}) + V_{15}(x,\overline{x})V_{16}(\overline{x}).
			\end{align}
			Observe that in this case
			\begin{equation*}
				\Eq\left[ e^{-r\tau_D}
				\I \{ \tau_D < \tau_{\overline{x}}^+ \}
				\I \{ K > e^{X_{\tau_D}}\}
				\right] = \E_x \left[ e^{-r\tau_{\overline{x}-c}^- }
				\I \{ \tau_{\overline{x}-c}^- < \tau_{\overline{x}}^+ \}
				\I \{ X_{\tau_{\overline{x}-c}^- } < \log(K)\}
				\right].
			\end{equation*}
			
			To calculate this expected value, we will use the following Gerber-Shiu measure
			\begin{equation*}
				K^{(r)}(a,x,\dx y, \dx z) =	\E_x \left[ e^{-r\tau_0^-} ;\ -X_{\tau_0^-} \in \dx y;\ X_{\tau_0^- -} \in \dx z; \tau_0^- < \tau_a^+ \right]
			\end{equation*}
			for $x,z \in [0,a]$ and $y \geq 0$. By \cite[Thm. 5.5]{kyprianou} we have
			\begin{equation*}
				K^{(r)}(a,x,\dx y, \dx z) = \frac{W^{(r)}(x)W^{(r)}(a-z) - W^{(r)}(a)W^{(r)}(x-z)}{W^{(r)}(a)}\Lambda(z + \dx y)\dx z.
			\end{equation*}
			Above, the Gerber-Shiu measure is associated with the first downward crossing of $0$.
			The proof of \cite[Thm. 5.5]{kyprianou} is given for the classical Cramer-Lundberg risk process, but it remains true without any changes for our L\'evy process
			$X_t$ given in \eqref{X_def}.
			We are, on the other hand, interested in the first time when the process starting from $x$ hits the level $\overline{x}-c$. Using the stationarity and independence of increments of L\'evy processes, we can shift our starting point from $x$ to $x+c-\overline{x}$ and thus we can search for the first time when $X_t$ becomes negative. Additionally, to make the condition $\{ X_{\tau_{\overline{x}-c}^- } < \log(K)\}$ true, we allow the undershoot $y$ of our shifted process to take values from set $(\overline{x}-\log(K)-c,\ \infty)$. Thus, we get that
			\begin{align*}
				\E&_x \left[ e^{-r\tau_{\overline{x}-c}^- }
				\I \{ \tau_{\overline{x}-c}^- < \tau_{\overline{x}}^+ \}
				\I \{ X_{\tau_{\overline{x}-c}^- } < \log(K)\}
				\right] \\
				\nonumber&=\! \int\limits_{(\overline{x}-\log(K)-c}^{\infty} \int\limits_{0}^{c}  \lambda \rho e^{-\rho (y+z)} \frac{W^{(r)}(x+c-\overline{x})W^{(r)}(c-z) - W^{(r)}(c)W^{(r)}(x+c-\overline{x}-z)}{W^{(r)}(c)}\dx z \dx y.
			\end{align*}
			Calculation of this double integral can be separated into three smaller single integrals. Firstly,
			\begin{equation*}
				\int\limits_{(\overline{x}-\log(K)-c}^{\infty} \rho e^{-\rho y} \dx y = e^{\rho (\log(K)+c - \overline{x})}.
			\end{equation*}
			Next,
			\begin{align*}
				\int\limits_{0}^{c}& e^{-\rho z} W^{(r)}(c-z) \dx z = \sum_{i=1}^3 C_i \int\limits_{0}^{c} e^{\gamma_i c - z(\gamma_i + \rho)} \dx z = \sum_{i=1}^3 C_i e^{\gamma_i c} \int\limits_{0}^{c} e^{ - z(\gamma_i + \rho)} \dx z \\
				\nonumber & = \sum_{i=1}^3 C_i e^{\gamma_i c} \frac{1 - e^{-c(\gamma_i+\rho)}}{\gamma_i + \rho}.
			\end{align*}
			Analogically, we have
			\begin{equation*}
				\int\limits_{0}^{c} e^{-\rho z} W^{(r)}(x+c-\overline{x}-z) \dx z = \int\limits_{0}^{x+c-\overline{x}} e^{-\rho z} W^{(r)}(x+c-\overline{x}-z) \dx z = \sum_{i=1}^3 C_i e^{\gamma_i c} \frac{1 - e^{-(x+c-\overline{x})(\gamma_i+\rho)}}{\gamma_i + \rho}
			\end{equation*}
			because $W^{(r)}(x) = 0$ for $x < 0$. Combining everything together, we get
			\begin{align*}
				\E&_x \left[ e^{-r\tau_{\overline{x}-c}^- }
				\I \{ \tau_{\overline{x}-c}^- < \tau_{\overline{x}}^+ \}
				\I \{ X_{\tau_{\overline{x}-c}^- } < \log(K)\}
				\right] = \lambda e^{\rho (\log(K)+c - \overline{x})} \\
				\nonumber &\times \sum_{i=1}^3 C_i e^{\gamma_i c} \left[ \frac{W^{(r)}(x+c-\overline{x})}{W^{(r)}(c)}\frac{1 - e^{-c(\gamma_i+\rho)}}{\gamma_i + \rho} - e^{\gamma_i(x - \overline{x})}\frac{1 - e^{-(x+c-\overline{x})(\gamma_i+\rho)}}{\gamma_i + \rho} \right].
			\end{align*}
			In order to calculate $\Eq\left[ e^{-r\tau_D} \left( K - e^{X_{\tau_D}} \right)^+ \I
			\{ \tau_D < \tau_{\overline{x}}^+ \} \right]$ we also need to consider the following equality
			\begin{align*}
				\E&_x \left[ e^{X_{\tau_{\overline{x}-c}^- }-r\tau_{\overline{x}-c}^- }
				\I \{ \tau_{\overline{x}-c}^- < \tau_{\overline{x}}^+ \}
				\I \{ X_{\tau_{\overline{x}-c}^- } < \log(K)\}
				\right] = e^x\E^\P_x \left[
				\I \{ \tau_{\overline{x}-c}^- < \tau_{\overline{x}}^+ \}
				\I \{ X_{\tau_{\overline{x}-c}^- } < \log(K)\}
				\right] \\
				&= e^x\tilde{\lambda} e^{\tilde{\rho} (\log(K)+c - \overline{x})} \sum_{i=1}^3 \tilde{C}_i e^{\tilde{\gamma}_i c} \left[ \frac{W^{\P}(x+c-\overline{x})}{W^{\P}(c)}\frac{1 - e^{-c(\tilde{\gamma}_i+\tilde{\rho})}}{\tilde{\gamma}_i + \tilde{\rho}} - e^{\tilde{\gamma}_i(x - \overline{x})}\frac{1 - e^{-(x+c-\overline{x})(\tilde{\gamma}_i+\tilde{\rho})}}{\tilde{\gamma}_i + \tilde{\rho}} \right] \\
				& = \frac{K\rho}{\rho + 1} \lambda e^{\rho (\log(K)+c - \overline{x})} \sum_{i=1}^3 C_i e^{\gamma_i c} \left[ \frac{W^{(r)}(x+c-\overline{x})}{W^{(r)}(c)}\frac{1 - e^{-c(\gamma_i+\rho)}}{\gamma_i + \rho} - e^{\gamma_i(x - \overline{x})}\frac{1 - e^{-(x+c-\overline{x})(\gamma_i+\rho)}}{\gamma_i + \rho} \right].
			\end{align*}
			Finally, observe that
			\begin{align}
				\nonumber V_{14}(x,\overline{x}) &= \Eq\left[ e^{-r\tau_D} \left( K - e^{X_{\tau_D}} \right)^+ \I
				\{ \tau_D < \tau_{\overline{x}}^+ \} \right]  =  K \E_x \left[ e^{-r\tau_{\overline{x}-c}^- }
				\I \{ \tau_{\overline{x}-c}^- < \tau_{\overline{x}}^+ \}
				\I \{ X_{\tau_{\overline{x}-c}^- } < \log(K)\}
				\right] \\
				\nonumber & - e^x \E^\P_x \left[
				\I \{ \tau_{\overline{x}-c}^- < \tau_{\overline{x}}^+ \}
				\I \{ X_{\tau_{\overline{x}-c}^- } < \log(K)\}
				\right] = \frac{K}{\rho+1} \lambda e^{\rho (\log(K)+c - \overline{x})} \\
				&\times \sum_{i=1}^3 C_i e^{\gamma_i c} \left[ \frac{W^{(r)}(x+c-\overline{x})}{W^{(r)}(c)}\frac{1 - e^{-c(\gamma_i+\rho)}}{\gamma_i + \rho} - e^{\gamma_i(x - \overline{x})}\frac{1 - e^{-(x+c-\overline{x})(\gamma_i+\rho)}}{\gamma_i + \rho} \right].\label{V14}
			\end{align}
			Moving on to the penultimate expected value in (\ref{V_3}), observe that
			\begin{equation}
				V_{15}(x,\overline{x}) = 	\Eq\left[ e^{-r\tau_{\overline{x}}^+} \I
				\{\tau_{\overline{x}}^+ < \tau_{\overline{x}+c}^- \} \right] = \frac{W^{(r)}(x+c-\overline{x})}{W^{(r)}(c)}\label{V15}
			\end{equation}
			as in (\ref{V2}) and (\ref{V4}).
			
			Now, similarly to (\ref{A7}), we consider the event
			\begin{align*}
				A_o^{16} = \{& \overline{X}_{\tau_D} \in \dx v,\ v \in (\overline{x}, \infty); D_{\tau_D-} \in \dx y,\ y \in \left[ 0,c \right); \\ \nonumber & D_{\tau_D} - c \in \dx h,\ h \in (v-c-\log(K), \infty)
				% h \sim \text{Exp}(\rho)
				\}.
			\end{align*}
			Thus, we can split the last expected value into
			\begin{equation*}
				V_{16}(\overline{x}) =	\E_{\overline{x}}\left[ e^{-r\tau_D} \left( K - e^{X_{\tau_D}} \right)^+ \right] = K\E_{\overline{x}}\left[ e^{-r\tau_D} \I \{ A_o^{16} \} \right] - e^{\overline{x}} \E^\P_{\overline{x}}\left[ \I \{ A_o^{16} \} \right].
			\end{equation*}
			
			We start by calculating the expected value under $\Q$
			\begin{align*}\label{V_K16}
				\E&_{\overline{x}}\left[ e^{-r\tau_D} \I \{ A_o^{16} \} \right] = \int\limits_{\overline{x}}^{\infty} \int\limits_{v-c-\log(K)}^{\infty}\int\limits_{0}^{c}F^\Q(v-(\log(K)+c))   R(r,\dx y) 	\Lambda(y-c-\dx h)   \dx v
				\\
				\nonumber&= \int\limits_{\overline{x}}^{\infty} \int\limits_{v-c-\log(K)}^{\infty}\int\limits_{0}^{c} \eta^\Q e^{-( v - \overline{x} )\eta^\Q} \left[ \frac{W^{(r)'}(y)}{\eta^{\Q}} - W^{(r)}(y) \right] \lambda\rho e^{\rho(y-c-h)} \dx y \dx h \dx v
				\\
				\nonumber&=  \int\limits_{\overline{x}}^{\infty} \eta^\Q  e^{-( v -\overline{x} )\eta^\Q}
				\int\limits_{v-c-\log(K)}^{\infty} \rho e^{ - \rho h} \dx h \dx v \int\limits_{0}^{c} \lambda e^{\rho(y-c)} \left[ \frac{W^{(r)'}(y)}{\eta^{\Q}} - W^{(r)}(y) \right] \dx y
				\\
				\nonumber&= \Gamma_Q \int\limits_{\overline{x}}^{\infty} \eta^\Q  e^{-( v - \overline{x} )\eta^\Q} e^{\rho(\log(K)+c-v)} \dx v  = \Gamma_Q e^{\rho (\log(K)+c) + \eta^\Q \overline{x}}\frac{\eta^\Q}{\eta^\Q + \rho} \left[ -e^{-v\left(\eta^\Q + \rho\right)} \right]^{\infty}_{\overline{x}} \\
				\nonumber&= \frac{\eta^\Q \Gamma_\Q}{\eta^\Q + \rho} e^{\rho( \log(K) + c - \overline{x} )}.	
			\end{align*}
			Redoing the calculation on the $\P$ measure and combining the results together we get
			\begin{equation}
				V_{16}(\overline{x}) = \frac{K}{\rho+1}\frac{\eta^\Q \Gamma_\Q}{\eta^\Q + \rho} e^{\rho( \log(K) + c - \overline{x} )}. \label{V16}
			\end{equation}
		\end{proof}
		
		%\newpage
		
		\subsection{Optimal stopping threshold $a^*$}
		To find the optimal level $a$ we choose $a^*$ satisfying condition \eqref{HJB7}, that is 
		\begin{equation}\label{optimalleveleq}\frac{\partial}{\partial x} V_a(x, \overline{x})\big|_{x = a^*}=-e^{a^*}.\end{equation}
		\begin{Thm}
			There exists a unique optimal level $a^*$ satisfying \eqref{optimalleveleq} and it solves the following equation
			\begin{equation}\label{otimala*}
				e^{a^*+c} - rK \frac{W^{(r)}(c)^2}{W^{(r)'}(c)} -  \frac{\eta^\Q e^{a^*}}{\eta^\Q-1}\left( \Delta^\Q + \frac{\rho\Gamma_\Q}{\rho+1} \right)  + \frac{Ke^{\eta^\Q(a^*-\log(K))}}{\eta^\Q-1}\left( \Delta^\Q + \frac{\rho\Gamma_\Q}{\rho+1} \right)
				=0,
			\end{equation}
			where $\eta^\Q$, $\Delta^\Q$, $\Gamma_\Q$ are given in (\ref{def_eta}), (\ref{def_Delta}), (\ref{def_Gamma}), respectively.
		\end{Thm}
		\begin{proof}
			We start the proof from observing that
			\begin{align*}
				\frac{\partial}{\partial x} V_a(x, \overline{x})%\big|_{x = a}
				&= \frac{\partial}{\partial x} V_1(x, \overline{x}) + (V_3(\overline{x}) +  V_4(\overline{x})( V_5 + V_6 V_7 ) )\frac{\partial}{\partial x} V_2(x, \overline{x}).
			\end{align*}
			Furthermore,
			\begin{equation*}
				\frac{\partial}{\partial x} V_1(x, \overline{x})  = K\left( rW^{(r)}(x-a) - Z^{(r)}(\overline{x}-a)\frac{W^{(r)'}(x-a)}{W^{(r)}(\overline{x}-a)} \right) - \left( e^x - \frac{e^{\overline{x}}W^{(r)'}(x-a)}{W^{(r)}(\overline{x}-a)} \right)
			\end{equation*}
			and
			\begin{equation*}
				\frac{\partial}{\partial x} V_2(x, \overline{x})  = \frac{W^{(r)'}(x-a)}{W^{(r)}(\overline{x}-a)}.
			\end{equation*}
			Note that $W^{(r)}(0) = 0$, $W^{(r)'}(0) = \frac{2}{\sigma^2}$, $Z^{(r)}(0) = 1$. Now, by setting $x = a$, we get
			\begin{equation*}
				\frac{\partial}{\partial x} V_1(x, \overline{x})\big|_{x = a}  = \frac{-2K  Z^{(r)}(\overline{x}-a)}{\sigma^2W^{(r)}(\overline{x}-a)}  -  e^a + \frac{2e^{\overline{x}}}{\sigma^2W^{(r)}(\overline{x}-a)}
			\end{equation*}
			and that
			\begin{equation*}
				\frac{\partial}{\partial x} V_2(x, \overline{x})\big|_{x = a}  = \frac{2}{\sigma^2W^{(r)}(\overline{x}-a)}.
			\end{equation*}
			Combining everything together we have
			\begin{align*}
				\frac{\partial}{\partial x} V(x, \overline{x})\big|_{x = a}  &= \frac{-2K  Z^{(r)}(\overline{x}-a)}{\sigma^2W^{(r)}(\overline{x}-a)}  -  e^a + \frac{2e^{\overline{x}}}{\sigma^2W^{(r)}(\overline{x}-a)} +
				\frac{2K  Z^{(r)}(\overline{x}-a)}{\sigma^2W^{(r)}(\overline{x}-a)} - \frac{2KZ^{(r)}(c)}{\sigma^2W^{(r)}(c)} \\
				\nonumber & - \frac{2e^{\overline{x}}}{\sigma^2W^{(r)}(\overline{x}-a)} + \frac{2e^{a+c}}{\sigma^2W^{(r)}(c)} + \frac{2}{\sigma^2W^{(r)}(c)}(V_5 + V_6 V_7) \\
				\nonumber & = -e^a + \frac{2e^{a+c}}{\sigma^2W^{(r)}(c)}\left(e^{a+c} - KZ^{(r)}(c) + V_5 + V_6 V_7 \right).
			\end{align*}
			In order to fulfill condition \eqref{optimalleveleq} (and (\ref{HJB7})) we search for $a^*$ satisfying
			\begin{equation}\label{simple_a}
				e^{a^*+c} - KZ^{(r)}(c) + V_5 + V_6 V_7 = 0.
			\end{equation}
			Now, using (\ref{delta_gamma}), we can rewrite the left-hand side of above equation as follows:
			\begin{align*}
				e^{a^*+c}& - KZ^{(r)}(c) + V_5 + V_6 V_7 = e^{a^*+c} - KZ^{(r)}(c) + Ke^{(a^*-\log(K))\eta^\Q}\frac{\Delta^\Q + \Gamma_\Q \frac{\rho+1-\eta^\Q}{\rho+1}}{\eta^\Q-1} \\
				\nonumber & - e^{a^*} \frac{\eta^\Q}{\eta^\Q-1}\left( \Delta^\Q + \frac{\rho}{\rho+1}\Gamma_\Q \right) + K\left( Z^{(r)}(c) - r\frac{W^{(r)}(c)^2}{W^{(r)'}(c)} \right)  \\
				\nonumber &+ \frac{K\eta^\Q\Gamma_\Q}{\left(\eta^\Q + \rho\right)(\rho+1)}e^{\eta^\Q(a^*-\log(K))} = e^{a^*+c} - rK \frac{W^{(r)}(c)^2}{W^{(r)'}(c)} -  \frac{\eta^\Q e^{a^*}}{\eta^\Q-1}\left( \Delta^\Q + \frac{\rho\Gamma_\Q}{\rho+1} \right)
				\\
				\nonumber & + Ke^{\eta^\Q(a^*-\log(K))}\left( \frac{\Delta^\Q + \Gamma_\Q}{\eta^\Q - 1} - \frac{\eta^\Q}{\eta^\Q-1} \frac{\Gamma_\Q}{\rho + 1} + \frac{\Gamma_\Q\eta^\Q}{(\eta^\Q+\rho)(\rho+1)}\right)
				\\
				\nonumber & = e^{a^*+c} - rK \frac{W^{(r)}(c)^2}{W^{(r)'}(c)} -  \frac{\eta^\Q e^{a^*}}{\eta^\Q-1}\left( \Delta^\Q + \frac{\rho\Gamma_\Q}{\rho+1} \right)  + \frac{Ke^{\eta^\Q(a^*-\log(K))}}{\eta^\Q-1}\left( \Delta^\Q + \frac{\rho\Gamma_\Q}{\rho+1} \right).
			\end{align*}
			This gives the equation \eqref{otimala*}.
			Using the intermediate value theorem we can prove that the solution of equation \eqref{simple_a} always exists. Indeed, observe that by taking $a^* \rightarrow -\infty$, the last expression above
			becomes $rK \frac{W^{(r)}(c)^2}{W^{(r)'}(c)}$, which is smaller than $0$. On the other hand, with $a^* = \log(K)$, we get
			\begin{align*}
				e^{a^*+c}& - KZ^{(r)}(c) + V_5 + V_6 V_7\big|_{a^* = \log(K)} = Ke^c - rK\frac{W^{(r)}(c)^2}{W^{(r)'}(c)} - \frac{K\eta^\Q}{\eta^\Q-1}\left( \Delta^\Q + \frac{\rho\Gamma_\Q}{\rho+1} \right)
				\\
				\nonumber &+ \frac{K}{\eta^\Q-1}\left( \Delta^\Q + \frac{\rho\Gamma_\Q}{\rho+1} \right) = K\left[ e^c - r\frac{W^{(r)}(c)^2}{W^{(r)'}(c)} - \Delta^\Q  \right.
				\\
				\nonumber &+ \left. \Gamma_\Q\left( \frac{\rho}{(\eta^\Q-1)(\eta^\Q+\rho)}  - \frac{\rho \eta^\Q}{(\eta^\Q-1)(\rho+1)} \right) \right] = K\left[ e^c - r\frac{W^{(r)}(c)^2}{W^{(r)'}(c)} - \Delta^\Q - \Gamma_\Q \right.
				\\ \nonumber &+ \left. \frac{\eta^\Q\Gamma_\Q}{(\eta^\Q+\rho)(\rho+1)} \right] = K\left[ e^c - Z^{(r)}(c) + \frac{\eta^\Q\Gamma_\Q}{(\eta^\Q+\rho)(\rho+1)} \right].
			\end{align*}
			Note that $Z^{(r)}(0) = 1$ for the function $Z^{(r)}(x)$ defined in \eqref{Z_def} and hence
			$\sum_{i=1}^{3} \frac{rC_i}{\gamma_1} = 1$. The function $Z^{(r)}(x)$ is then a weighted average of three exponential functions with $C_2, C_3, \gamma_2, \gamma_3 < 0$. Therefore, using fact that $\gamma_1=1$ (see \eqref{eta2}),
			we have $ e^c - Z^{(r)}(c) > 0$ and hence above expression is strictly positive for $a^* = \log(K)$. Thus indeed solution $a^*$ always exists.
		\end{proof}
		
		%\newpage
		\subsection{Proof of Theorem \ref{thm2}}
		According to Remark \ref{uwaga2} it is sufficient to show that $\hat{V}(x, \overline{x})=V_{a^*}(x, \overline{x})$ is in the domain of the infinitesimal generator
		$\mathcal{L}$, that is, that $\hat{V}(x, \overline{x})\in \mathcal{C}^2_0(\mathbb{R})$ and that the boundary condition
		\eqref{gen_dom_cut} is satisfied, that is, that
		\[\frac{\partial}{\partial \overline{x}} V_{a^*}(x, \overline{x})=0\quad \text{for}\quad  x = \overline{x}.\]
		Furthermore, we have to verify that all the conditions given in the verification Lemma \ref{HJB} are satisfied.
		
		The fact that $\hat{V}(x, \overline{x})\in \mathcal{C}^2(\mathbb{R})$ follows from Theorem \ref{valuea} and from the form of the scale functions $W^{(r)}(x)$ and $Z^{(r)}(x)$ given
		in \eqref{W_sum} and \eqref{Z_def}. %and the functions $W^{\P}(x)$ and $Z^{\P}(x)$ given  in \eqref{W_P} and \eqref{Z_P}, respectively.
		The fact that $\hat{V}(x, \overline{x})$ disappear as $x, \overline{x}$ tend to infinity follows directly from the definition of
		$\hat{V}(x, \overline{x})=V_{a^*}(x, \overline{x})$.
		We will now verify that the condition \eqref{gen_dom_cut} holds true.
		
		We start from the case $\overline{x} < a^*+c$ for which we have
		\begin{align*}
			\frac{\partial}{\partial \overline{x}} \hat{V}(x, \overline{x}) &= \frac{\partial}{\partial \overline{x}} V_1(x, \overline{x}) + V_3(\overline{x})\frac{\partial}{\partial \overline{x}} V_2(x, \overline{x}) + V_4(\overline{x})( V_5 + V_6 V_7 ) \frac{\partial}{\partial \overline{x}} V_2(x, \overline{x}) \\
			\nonumber &+ V_2(x,\overline{x}) \frac{\partial}{\partial \overline{x}} V_3(\overline{x})+ V_2(x, \overline{x})( V_5 + V_6 V_7 )\frac{\partial}{\partial \overline{x}} V_4(\overline{x}).
		\end{align*}
		Note that		
		\begin{equation*}
			\frac{\partial}{\partial \overline{x}} V_2(x, \overline{x}) = \frac{-W^{(r)'}(\overline{x}-a^*)W^{(r)}(x-a^*)}{W^{(r)}(\overline{x}-a^*)^2}
		\end{equation*}		
		and that		
		\begin{equation*}
			\frac{\partial}{\partial \overline{x}} V_4(\overline{x}) = \frac{-W^{(r)'}(\overline{x}-a^*)}{W^{(r)}(c)}.
		\end{equation*}				
		Therefore
		\begin{equation*}
			V_4(\overline{x})\frac{\partial}{\partial \overline{x}} V_2(x, \overline{x}) +  V_2(x, \overline{x})\frac{\partial}{\partial \overline{x}} V_4(\overline{x}) = 0.
		\end{equation*}
		Additionally, we have
		\begin{align*}
			\frac{\partial}{\partial \overline{x}} V_1(x, \overline{x}) &=	K W^{(r)}(x-a^*) \frac{Z^{(r)'}(\overline{x}-a^*)W^{(r)}(\overline{x}-a^*)-W^{(r)'}(\overline{x}-a^*)Z^{(r)}(\overline{x}-a^*)}{W^{(r)}(\overline{x}-a^*)^2} \\
			\nonumber &+ e^{\overline{x}} \frac{W^{(r)}(x-a^*)W^{(r)}(\overline{x}-a^*) - W^{(r)'}(\overline{x}-a^*)W^{(r)}(x-a^*)}{W^{(r)}(\overline{x}-a^*)^2} \\
			\nonumber &= \frac{1}{W^{(r)}(\overline{x}-a^*)^2} \left[ W^{(r)}(\overline{x}-a^*)W^{(r)}(x-a^*)\left( e^{\overline{x}} - rKW^{(r)}(\overline{x}-a^*) \right) \right. \\
			\nonumber & \left. + W^{(r)}(x-a^*)W^{(r)'}(\overline{x}-a^*) \left( KZ^{(r)}(\overline{x}-a^*) - e^{\overline{x}} \right) \right]
		\end{align*}		
		and that
		\begin{equation*}
			\frac{\partial}{\partial \overline{x}} V_3(\overline{x}) = rK W^{(r)}(\overline{x}-a^*) - K\frac{Z^{(r)}(c)}{W^{(r)}(c)}W^{(r)'}(\overline{x}-a^*)  - e^{\overline{x}} + \frac{e^{a^*+c}}{W^{(r)}(c)}W^{(r)'}(\overline{x}-a^*).
		\end{equation*}	
		Finally, putting all terms together,
		\begin{align}
			\frac{\partial}{\partial \overline{x}} \hat{V}(x, \overline{x})=\frac{\partial}{\partial \overline{x}}& V_1(x, \overline{x}) + V_3(\overline{x})\frac{\partial}{\partial \overline{x}} V_2(x, \overline{x}) +  V_2(x, \overline{x})\frac{\partial}{\partial \overline{x}} V_3(\overline{x}) =  \frac{-W^{(r)'}(\overline{x}-a^*)W^{(r)}(x-a^*)}{W^{(r)}(\overline{x}-a^*)^2}
			\\
			\nonumber &\times \left[ K\left( Z^{(r)}(\overline{x}-a^*) - \frac{Z^{(r)}(c)}{W^{(r)}(c)}W^{(r)}(\overline{x}-a^*)\right) - e^{\overline{x}}  + \frac{e^{a^*+c}W^{(r)}(\overline{x}-a^*)}{W^{(r)}(c)}  \right] \\
			\nonumber & + \frac{W^{(r)}(x-a^*)}{W^{(r)}(\overline{x}-a^*)} \left[ rKW^{(r)}(\overline{x}-a^*) - K\frac{Z^{(r)}(c)}{W^{(r)}(c)}W^{(r)'}(\overline{x}-a^*) - e^{\overline{x}} + \frac{e^{a^*+c}}{W^{(r)}(c)}W^{(r)'}(\overline{x}-a^*) \right] \\
			\nonumber & + \frac{1}{W^{(r)}(\overline{x}-a^*)^2} \left[ W^{(r)}(\overline{x}-a^*)W^{(r)}(x-a^*)\left( e^{\overline{x}} - rKW^{(r)}(\overline{x}-a^*) \right) \right. \\
			\nonumber & \left. + W^{(r)}(x-a^*)W^{(r)'}(\overline{x}-a^*) \left( KZ^{(r)}(\overline{x}-a^*) - e^{\overline{x}} \right) \right].
		\end{align}
		
		When we set $x = \overline{x}$, all terms cancel out and the condition \eqref{gen_dom_cut} is fulfilled.
		
		%\newpage

		We consider now the case of $a^*+c < \overline{x} < \log(K) + c$. In this case we have
		\begin{align*}
			\frac{\partial}{\partial \overline{x}} \hat{V}(x, \overline{x}) &= \frac{\partial}{\partial \overline{x}} V_{10}(x, \overline{x}) + (V_{12}(\overline{x}) + V_{13}(\overline{x})V_7)\frac{\partial}{\partial \overline{x}} V_{11}(x, \overline{x}) + V_{11}(x, \overline{x})\left(\frac{\partial}{\partial \overline{x}}V_{12}(\overline{x}) + 	\frac{\partial}{\partial \overline{x}}V_{13}(\overline{x})V_7\right).
		\end{align*}		
		Further,
		\begin{align*}
			\frac{\partial}{\partial \overline{x}} V_{10}(x, \overline{x})\big|_{x = \overline{x}} = -rKW^{(r)}(c) + \eta^\Q \left( KZ^{(r)}(c) - e^{\overline{x}} \right) + e^{\overline{x}}
		\end{align*}
		and
		\begin{equation*}
			\frac{\partial}{\partial \overline{x}} V_{11}(x, \overline{x})\big|_{x = \overline{x}} = - \frac{W^{(r)'}(c)}{W^{(r)}(c)} = -\eta^\Q, \quad V_{11}(\overline{x}, \overline{x}) = 1.
		\end{equation*}
		Moreover, we have
		\begin{equation*}
			\frac{\partial}{\partial \overline{x}}V_{13}(\overline{x}) = \eta^\Q e^{-\eta^\Q (\log(K) + c - \overline{x})} = \eta^\Q V_{13}(\overline{x})
		\end{equation*}
		and
		\begin{equation*}
			V_7 V_{13}(\overline{x})\frac{\partial}{\partial \overline{x}} V_{11}(x, \overline{x})\big|_{x = \overline{x}} -  V_7V_{11}(\overline{x}, \overline{x})\frac{\partial}{\partial \overline{x}}V_{13}(\overline{x}) = 0.
		\end{equation*}
		Furthermore,		
		\begin{align*}
			\frac{\partial}{\partial \overline{x}}V_{12}(\overline{x}) &= K\frac{\eta^\Q}{\eta^\Q-1}\left( \Delta^\Q + \Gamma_\Q \right)e^{-(\log(K)+c-\overline{x})\eta^\Q}   - \frac{\eta^\Q}{\eta^\Q-1}\left[ e^{\overline{x}-c}\Delta^\Q + \frac{\Gamma_Q}{\rho+1}\left( \rho e^{\overline{x}-c} +  K\eta^\Q e^{-(\log(K)+c-\overline{x})\eta^\Q} \right)  \right]
			\\
			\nonumber & = \frac{\eta^\Q}{\eta^\Q-1} \left[ Ke^{-(\log(K)+c-\overline{x})\eta^\Q}\left( \Delta^\Q + \Gamma_\Q\left( 1 - \frac{\eta^\Q}{\rho+1} \right) \right) - e^{\overline{x}-c}\left( \Delta^\Q + \frac{\rho}{\rho+1}\Gamma_\Q \right) \right].
		\end{align*}
		This gives
		\begin{align*}
			\frac{\partial}{\partial \overline{x}} \hat{V}(x, \overline{x}) =\frac{\partial}{\partial \overline{x}} &V_{10}(x, \overline{x})\big|_{x = \overline{x}} + V_{12}(\overline{x})\frac{\partial}{\partial \overline{x}} V_{11}(x, \overline{x})\big|_{x = \overline{x}} + V_{11}(\overline{x}, \overline{x})\frac{\partial}{\partial \overline{x}}V_{12}(\overline{x})
			\\
			\nonumber & = -rKW^{(r)}(c) + e^{\overline{x}} + \eta^\Q \left[ KZ^{(r)}(c) - e^{\overline{x}} - K\left(
			\Delta^\Q + \Gamma_Q \right) +  e^{\overline{x}-c}\left( \Delta^\Q + \frac{\rho}{\rho+1}\Gamma_Q \right) \right]
			\\
			\nonumber & = -rKW^{(r)}(c) + e^{\overline{x}} + \eta^\Q \left[ K\left( Z^{(r)}(c) - Z^{(r)}(c) + \frac{rW^{(r)}(c)}{\eta^\Q} \right) - e^{\overline{x}}  +  e^{\overline{x}-c}\left( \Delta^\Q + \frac{\rho}{\rho+1}\Gamma_Q \right) \right]
			\\
			\nonumber & = e^{\overline{x}}\left(1 - \eta^\Q\right) + \eta^\Q e^{\overline{x}-c}\left( \Delta^\Q + \frac{\rho}{\rho+1}\Gamma_Q \right) = e^{\overline{x}} \left[ 1 - \eta^\Q + \eta^\Q e^{-c}\left( \Delta^\Q + \frac{\rho}{\rho+1}\Gamma_Q \right) \right]
			\\
			\nonumber & = e^{\overline{x}} \left[ 1 - \eta^\Q + (\eta^\Q-1)\left( \frac{\eta^\Q e^{-c}}{\eta^\Q-1} \Delta^\Q + \frac{\rho}{\rho+1}\frac{\eta^\Q  e^{-c}}{\eta^\Q-1} \Gamma_Q \right) \right] = e^{\overline{x}} \left[ 1 - \eta^\Q + (\eta^\Q-1)\left( \Delta^\P + \Gamma_\P \right) \right]
			\\
			\nonumber & = e^{\overline{x}} \left[ 1 - \eta^\Q + (\eta^\Q-1)\left( \Delta^\P + \Gamma_\P \right) \right] = e^{\overline{x}} (1 - \eta^\Q + \eta^\Q-1) = 0.
		\end{align*}
		Thus condition \eqref{gen_dom_cut} is satisfied.
		
		We analyse now the remaining case of $\log(K) + c < \overline{x}$.
		In this case we have
		\begin{align*}
			\frac{\partial}{\partial \overline{x}} \hat{V}(x, \overline{x}) &= \frac{\partial}{\partial \overline{x}} V_{14}(x, \overline{x}) +  V_{15}(x, \overline{x}) \frac{\partial}{\partial \overline{x}} V_{16}(\overline{x}) + V_{16}(\overline{x}) \frac{\partial}{\partial \overline{x}}V_{15}(x, \overline{x}).
		\end{align*}		
		Further,
		\begin{align*}
			\frac{\partial}{\partial \overline{x}}& V_{14}(x, \overline{x}) = -\rho V_{14}(x, \overline{x}) +	\frac{K}{\rho+1} \lambda e^{\rho (\log(K)+c - \overline{x})} \\
			&\times \sum_{i=1}^3 C_i e^{\gamma_i c} \left[ \frac{-W^{(r)'}(x+c-\overline{x})}{W^{(r)}(c)}\frac{1 - e^{-c(\gamma_i+\rho)}}{\gamma_i + \rho} - \frac{-\gamma_i e^{\gamma_i(x - \overline{x})} - \rho e^{-\gamma_i c - \rho(x+c) + \rho\overline{x}}}{\gamma_i + \rho} \right]
		\end{align*}
		Using (\ref{sum_Ci_0}) we get
		\begin{align*}
			\frac{\partial}{\partial \overline{x}}& V_{14}(x, \overline{x})\big|_{x = \overline{x}} = \frac{K \lambda e^{\rho (\log(K)+c - \overline{x})}}{\rho+1}\sum_{i=1}^3 \frac{C_i e^{\gamma_i c}}{\gamma_i + \rho}\left[- \eta^\Q + \gamma_i +  (\rho + \eta^\Q)e^{-c(\gamma_i + \rho)} \right]
			\\
			& = \frac{K \lambda e^{\rho (\log(K)+c - \overline{x})}}{\rho+1} \left[ \eta^\Q\sum_{i=1}^3 \frac{C_i e^{\gamma_i c}}{\gamma_i + \rho}\left( \frac{\gamma_i}{\eta^\Q} - 1 \right) + \sum_{i=1}^3 \frac{C_i(\rho + \eta^\Q)e^{-c\rho}}{\gamma_i + \rho} \right]
			\\
			& = \frac{K\eta^\Q e^{\rho (\log(K)+c - \overline{x})}}{\rho + 1}\Gamma_\Q = (\eta^\Q + \rho) V_{16}(\overline{x}).
		\end{align*}
		Additionally, we have
		\begin{equation*}
			\frac{\partial}{\partial \overline{x}}V_{15}(x, \overline{x}) = \frac{-W^{(r)'}(x + c - \overline{x})}{W^{(r)}(c)}, \quad \frac{\partial}{\partial \overline{x}}V_{15}(x, \overline{x})\big|_{x = \overline{x}} = -\eta^\Q
		\end{equation*}		
		and
		\begin{equation*}
			\frac{\partial}{\partial \overline{x}}V_{16}(\overline{x}) = -\rho V_{16}(\overline{x}).
		\end{equation*}
		This gives		
		\begin{equation*}
			\frac{\partial}{\partial \overline{x}} \hat{V}(x, \overline{x})\big|_{x = \overline{x}} = (\eta^\Q + \rho) V_{16}(\overline{x}) - \eta^\Q V_{16}(\overline{x}) -\rho V_{16}(\overline{x}) = 0
		\end{equation*}	
		which completes the proof of condition \eqref{gen_dom_cut}.

		In the final step, we will verify that all the conditions of the verification Lemma \ref{HJB} are satisfied.
		
		First, observe that the scale functions $W^{(r)}(x)$ and $Z^{(r)}(x)$ defined as in (\ref{W_sum}) and (\ref{Z_def}) fulfill condition (\ref{HJB1}). Additionally, $e^x$ also satisfies this condition, since $\gamma_1 = 1$. Now, if $x > a^*$, then for all three cases specified in Theorem \ref{valuea}, function $V_{a^*}(x,\overline{x})$ is just a combination of $W^{(r)}(x)$, $Z^{(r)}(x)$ and $e^x$. Thus, condition (\ref{HJB1}) holds.
		
		Now, from the fourth case specified in Theorem \ref{valuea} we can see that $(\mathcal{L} V_{a^*} - rV_{a^*})(x, \overline{x}) = -rK$ when $x \leq a^*$. Therefore, condition (\ref{HJB1b}) is satisfied as well. 
The form of function $V_{a^*}(x, \overline{x})$ defined in Theorem \ref{valuea} makes equation (\ref{HJB3}) directly to be satisfied for $x \leq a^*$.
		
		Let us first handle the smooth paste conditions before moving on to the fourth condition of the HJB system. Note that equality (\ref{HJB7}) is true due to our choice of $a^*$ solving (\ref{simple_a}). Furthermore, note that $V_1(a^*,\overline{x}) = K - e^{a^*}$ 
and $V_2(a^*,\overline{x}) = 0$, which implies that condition (\ref{HJB6}) is satisfied.
		
		Finally, to prove (\ref{HJB4}), we need to ensure that $V_{a^*}(x,\overline{x})$ dominates over the gain function.
		
Proposition \ref{thm1} implies that the stopping region is reached by $X_t$ process from above. On the other hand, from conditions (\ref{HJB6}) and (\ref{HJB7}) we know that $\tau_{a^*}^-$ is the first time when the candidate for the value function equals its payoff. %We need to make sure that it dominates the gain for $x > a$.
Let us consider the pair $(x, \overline{x})=(a^*+c, a^*+c)$. Note that in this case, we have $V_1(a^*+c, a^*+c) = 0$, $V_2(a^*+c, a^*+c) = 1$, $V_3(a^*+c) = 0$ and $V_4(a^*+c) = 1$. Therefore, from (\ref{simple_a}) we get that $V_{a^*}(a^*+c, a^*+c) =  KZ^{(r)}(c) - e^{a^*+c}$. 
On the other hand, the immediate payout if equal to $K - e^{a^*+c}$. By (\ref{Z_def}) it immediately follows that $Z(x) > 1$ for all $x>0$ which makes condition (\ref{HJB4}) satisfied.
		
		\hfill$\square$	
		%\newpage
		\section{Numerical analysis}\label{sec:num}
		
		In this section, we analyze several properties of options capped by drawdown. In Figure \ref{Smooth_paste}, we present the smooth-pasting of the payoff and value functions, defined by equations (\ref{HJB6}) and (\ref{HJB7}). The term "Projection", mentioned in the legend, refers to the form of the function $V$, defined on the continuation region, applied to the stopping region. The chart confirms that the smooth-paste condition indeed holds.
		
		Next, Figure \ref{Price} illustrates how the option price depends on the initial values $X_0 = x$ and $\overline{X}_0 = \overline{x}$. Half of the chart is set to zero due to domain limitation $x \leq \overline{x}$. The plot reveals an interesting behavior of the function, explained below, when both $x$ and $\overline{x}$ are sufficiently high. 
Figure \ref{PriceCU} shows the same function, zoomed into the lower-right corner.
		
		The first notable feature is the non-differentiability of the function at $\overline{x} = \log(K) + c$, when $x \in (\log(K), \log(K) + c)$. The explanation is straightforward: if $\overline{x} < \log(K) + c$, a non-zero payout can still be achieved via diffusion, whereas for $\overline{x} > \log(K) + c$, it can only result from a Poissonian jump. Consequently, the structure of the value function changes, leading to non-differentiability.
		
		Additionally, we observe that for $\overline{x} > \log(K) + c$, the value function is zero when $x < \overline{x} - c$, which is expected, as the option is out of the money and immediately stopped by the drawdown. However, for $x > \overline{x} - c$, the value function increases with respect to $x$, which might be surprising, since for $\overline{x} < \log(K) + c$ it behaves in the opposite manner. This can be explained by the fact that the higher the value of $x$, the longer it takes for the drawdown to occur. As a result, the probability of observing a downward jump that causes $x$ to fall below $\log(K)$ increases.
		
		Next, we perform a sensitivity analysis of both the stopping barrier and the option price with respect to the volatility $\sigma$, the risk-free rate $r$, the jump size parameter $\rho$, and the jump frequency parameter $\lambda$.
		
		In Figure \ref{A_r_sigma}, we examine the optimal barrier $e^{a^*}$ of the underlying asset price process. It is evident that the barrier increases with higher values of $r$ and lower values of $\sigma$, indicating that an increase in the drift parameter of $X_t$ results in an upward shift of the barrier.
		
		Figure \ref{V_r_sigma} presents the sensitivities of the option price with respect to the risk-free rate and volatility. It is clear that the highest price is obtained for the lowest values of $r$ and $\sigma$. However, it is also evident that for a high value of $r = 0.5$, the value function increases with $\sigma$. This suggests that the impact of volatility on the option price depends on the values of the other parameters.
		
		A similar analysis is conducted for $\rho$ and $\lambda$. Figure \ref{A_rho_lambda} shows that the optimal stopping barrier decreases with increasing $\lambda$ and decreasing $\rho$. This is intuitive, as a smaller $\rho$ implies a larger mean jump size. Likewise, the option price, presented in Figure \ref{V_rho_lambda}, is highest for the largest value of $\lambda$ and the smallest value of $\rho$, which is sensible, since in such a scenario the likelihood of the process $X_t$ dropping significantly below $\log(K)$ increases. Figures \ref{V_rho} and \ref{V_lambda} present side views of the chart from Figure \ref{V_rho_lambda}. There, we observe that the option price grows exponentially as $\rho$ decreases and linearly as $\lambda$ increases. This behavior can be explained by the fact that while $\lambda$ controls the frequency of jumps, $\rho$ exponentially influences the jump sizes, affecting the payout after a jump occurs.
		
		\begin{figure}
			\includegraphics[width=\textwidth]{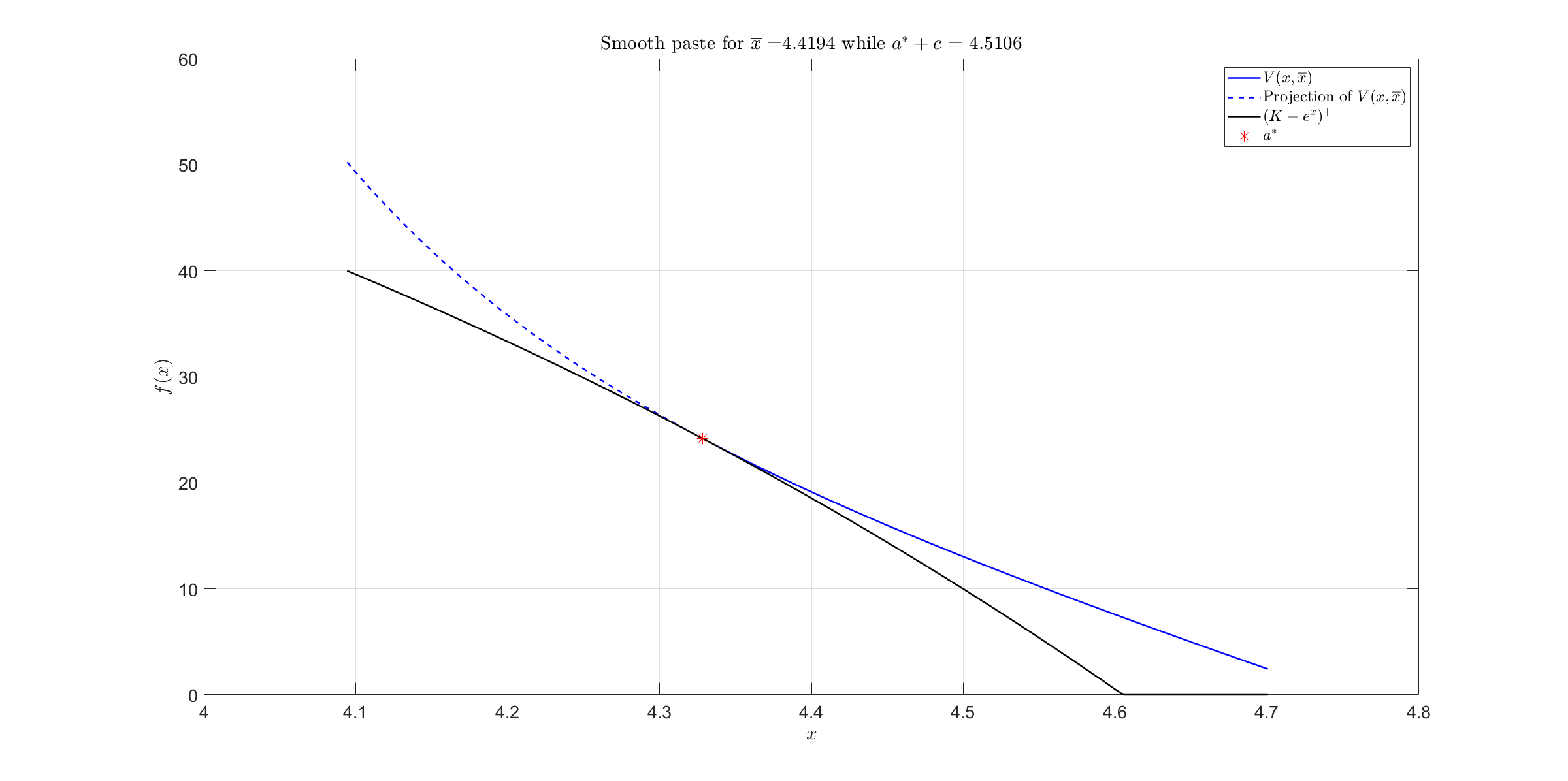}
			\caption{Smooth paste of the option price $V$ and the option payoff. Parameters of the model: $r = 0.1,\ \sigma = 0.2,\ e^c = 1.2,\ \rho = 3,\ \lambda = 0.2,\ K = 100$.}
			\label{Smooth_paste}
		\end{figure}

		\begin{figure}
			\includegraphics[width=\textwidth]{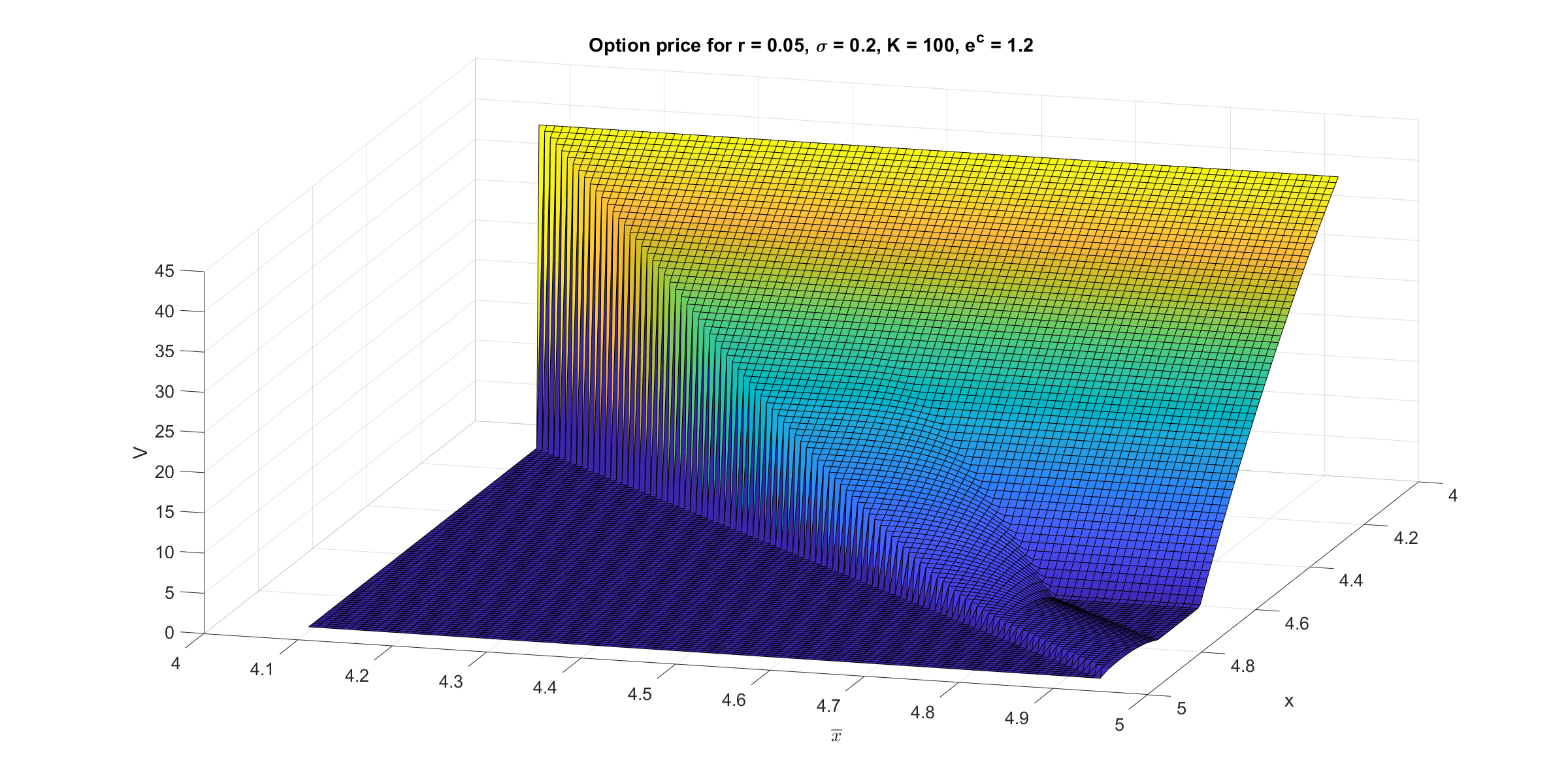}
			\caption{Option price depending on $x$ and $\overline{x}$. Note that the function is not defined for $\overline{x} > x$.}
			\label{Price}
		\end{figure}
		
		\begin{figure}
			\includegraphics[width=\textwidth]{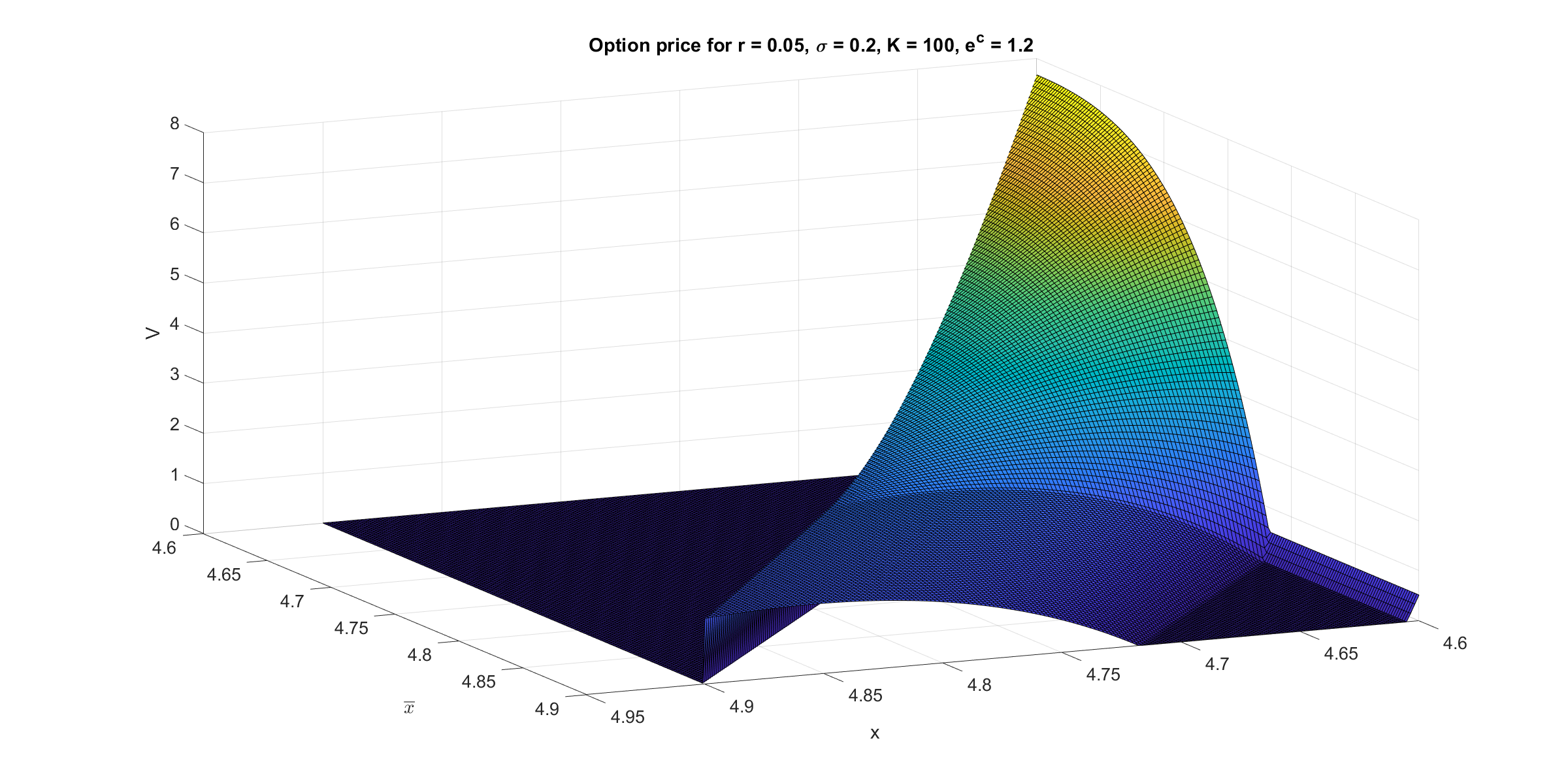}
			\caption{Option price depending on $x$ and $\overline{x}$. Here, non-differentiability is clearly visible for $s = \log(K) + c$.}
			\label{PriceCU}
		\end{figure}
		
		\begin{figure}
			\includegraphics[width=\textwidth]{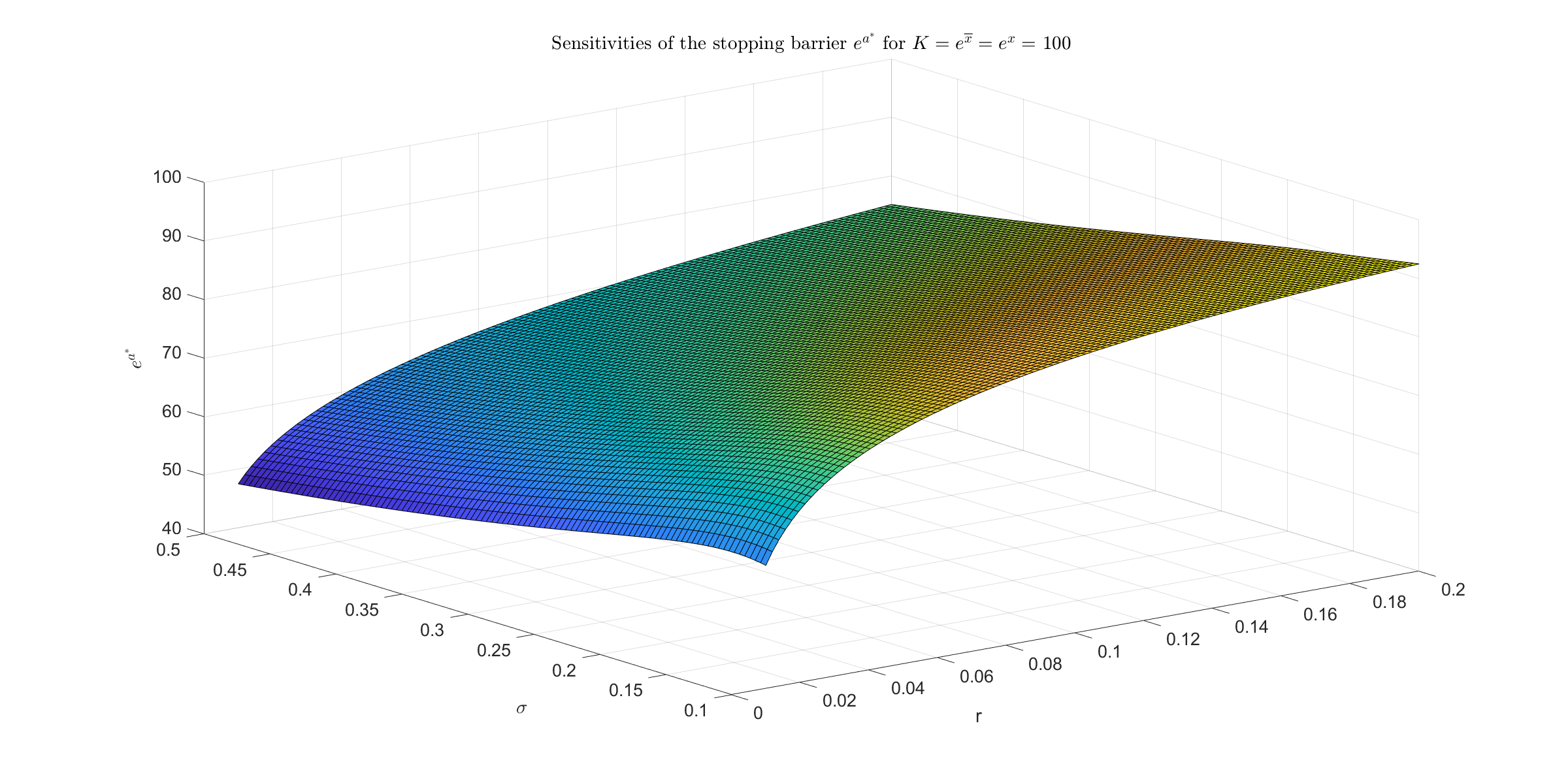}
			\caption{Stopping barrier $e^{a^*}$ of underlying asset price process $S_t$ depending on the risk-free rate $r$ and the volatility $\sigma$.}
			\label{A_r_sigma}
		\end{figure}
		
		\begin{figure}
			\includegraphics[width=\textwidth]{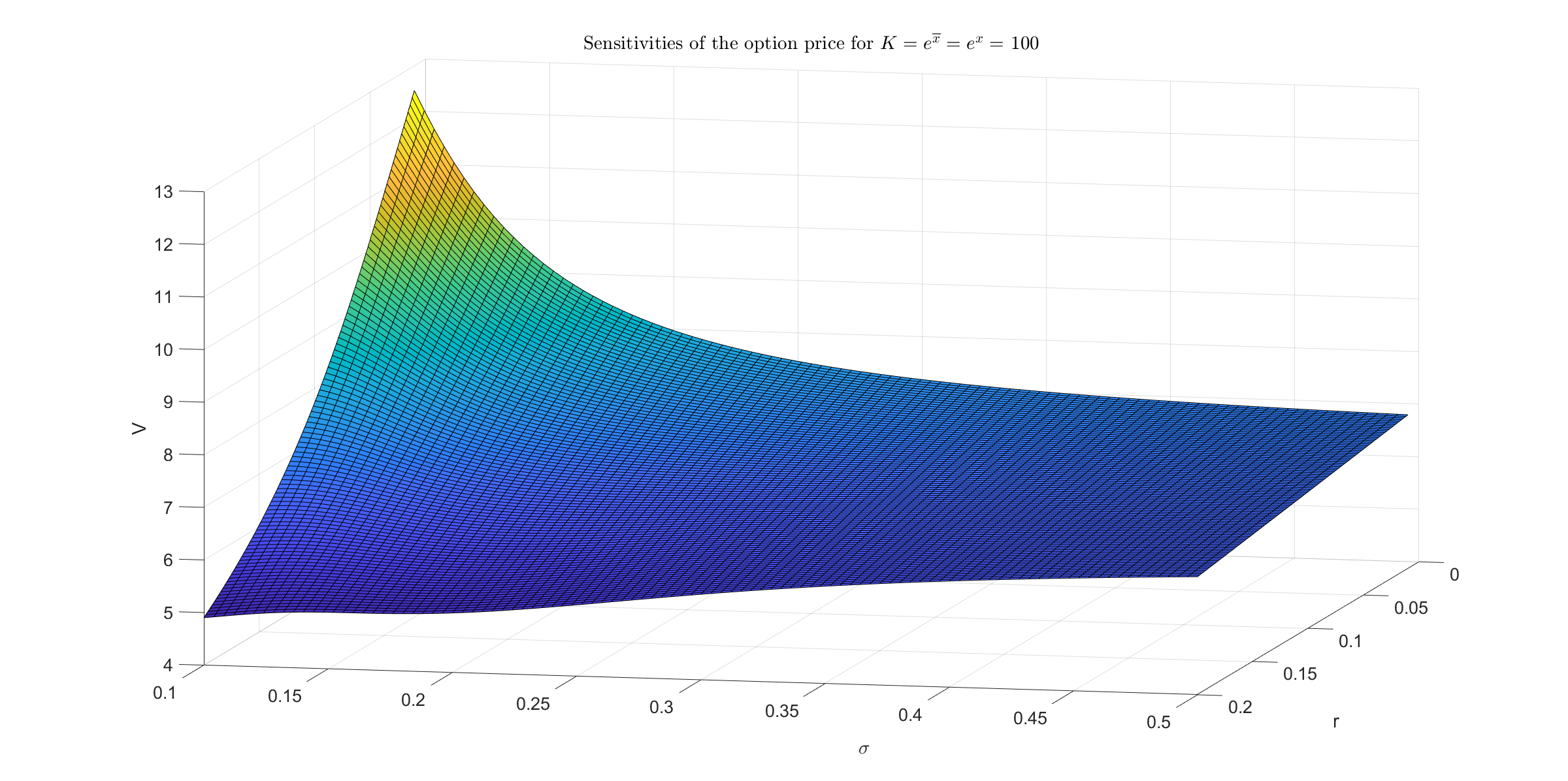}
			\caption{Sensitivities of the option price depending on the risk-free rate $r$ and the volatility $\sigma$.}
			\label{V_r_sigma}
		\end{figure}
		
		\begin{figure}
			\includegraphics[width=\textwidth]{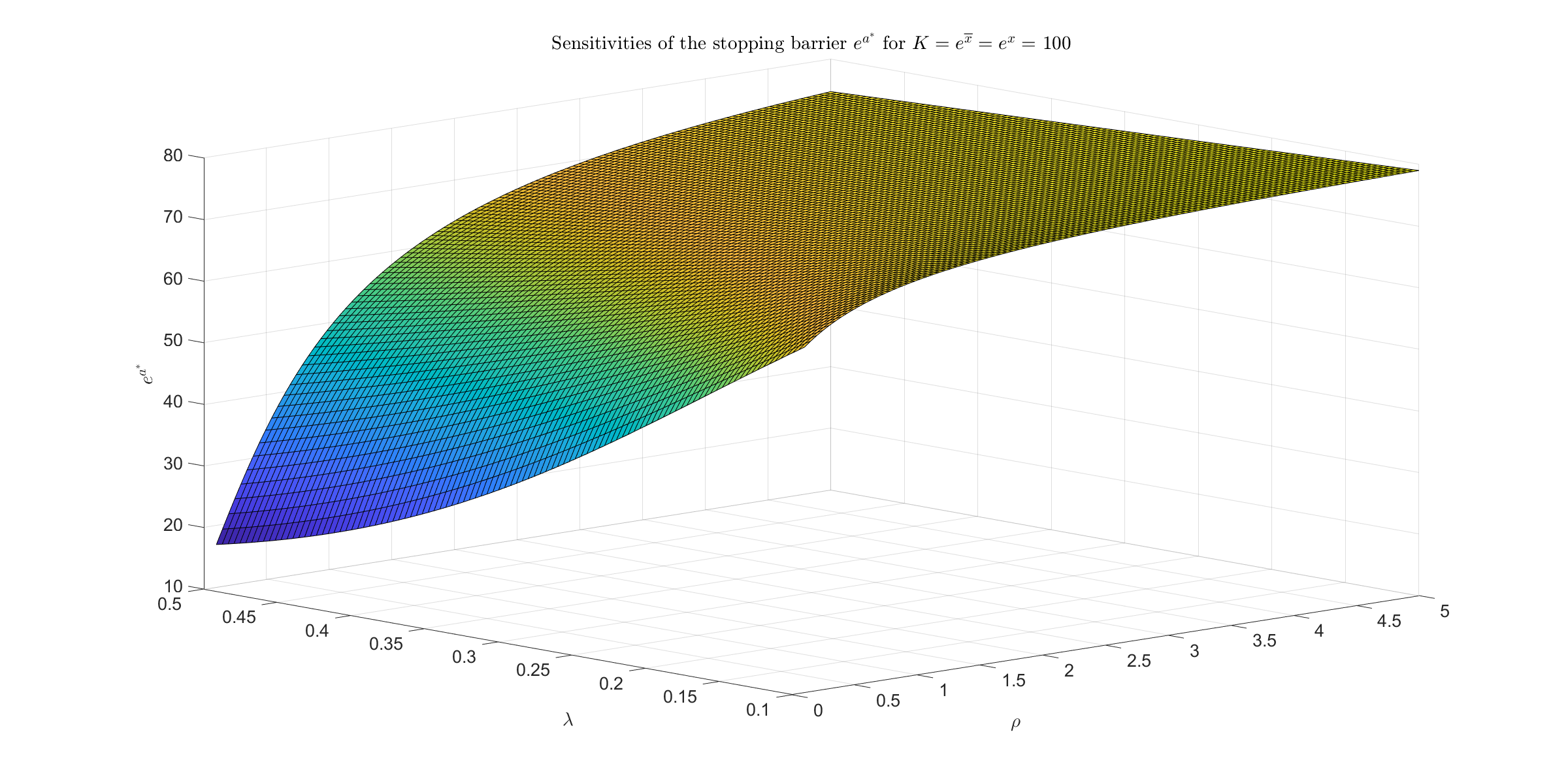}
			\caption{Stopping barrier $e^{a^*}$ of underlying asset price process $S_t$ depending on the jump size $\rho$ and the jump intensity $\lambda$.}
			\label{A_rho_lambda}
		\end{figure}
		
		\begin{figure}
			\includegraphics[width=\textwidth]{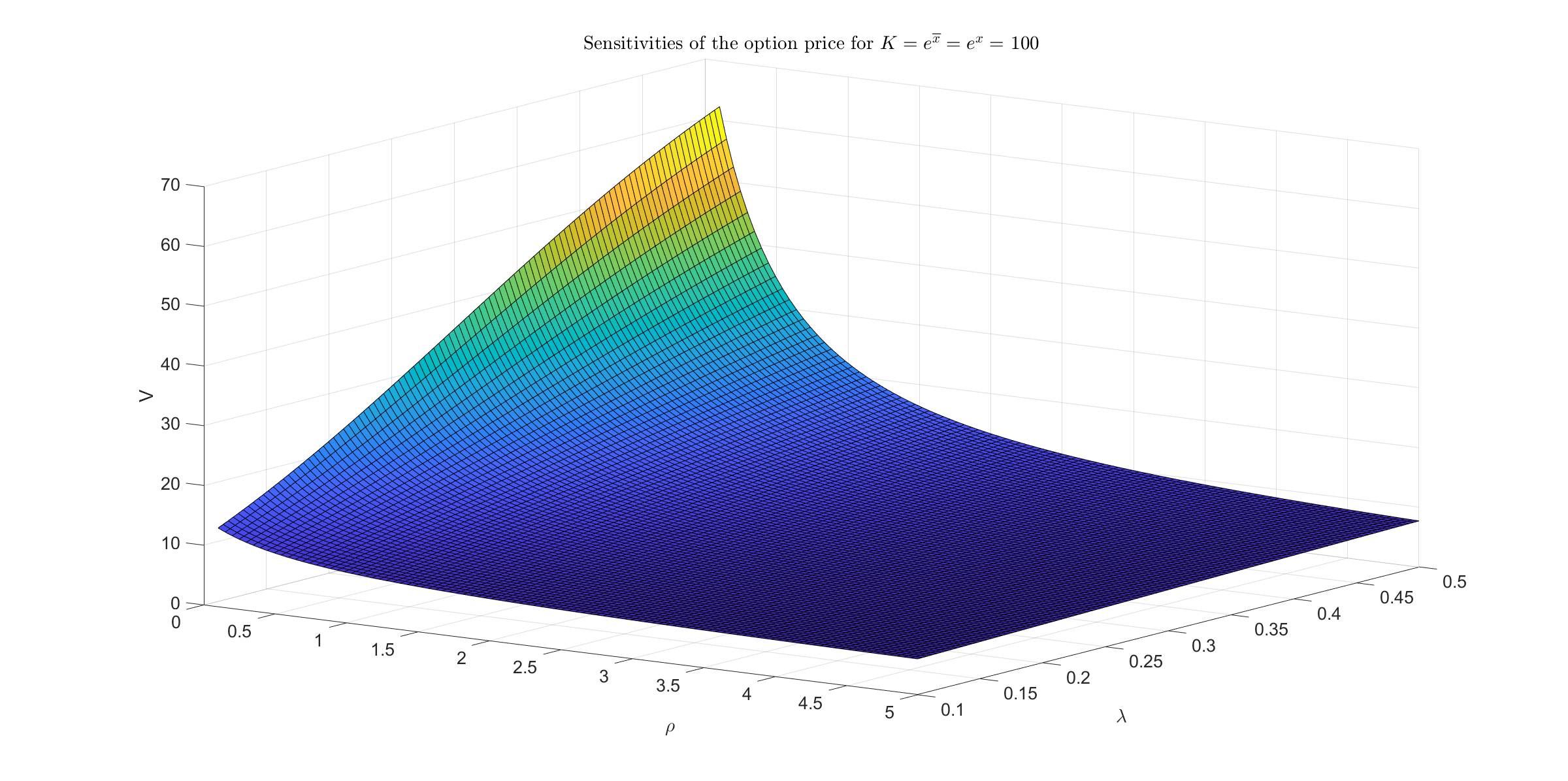}
			\caption{Sensitivities of the option price depending on the jump size $\rho$ and the jump intensity $\lambda$.}
			\label{V_rho_lambda}
		\end{figure}
		
		\begin{figure}
			\includegraphics[width=\textwidth]{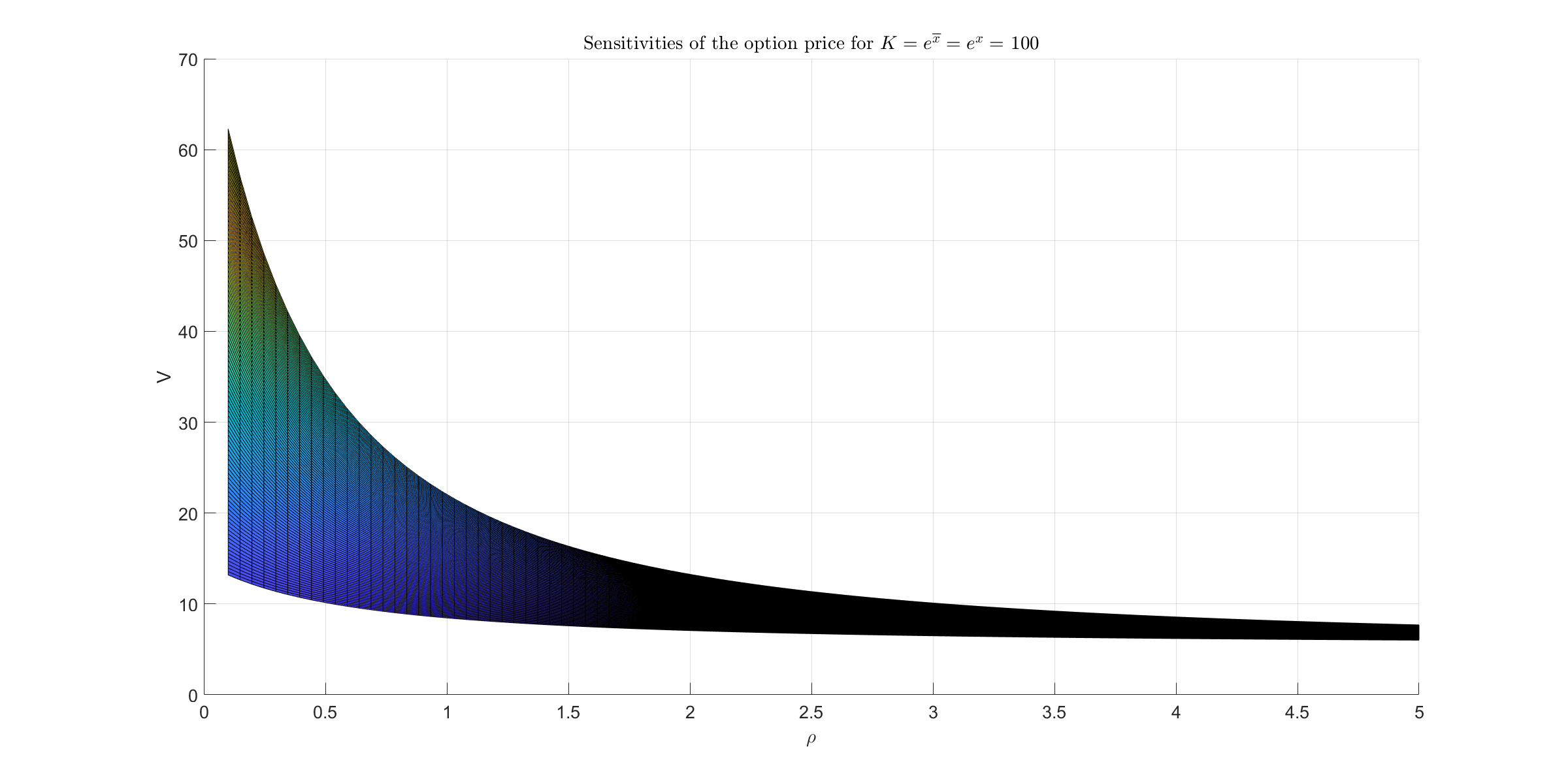}
			\caption{Sensitivities of the option price depending on the jump size $\rho$ and the jump intensity $\lambda$.}
			\label{V_rho}
		\end{figure}
		
		\begin{figure}
			\includegraphics[width=\textwidth]{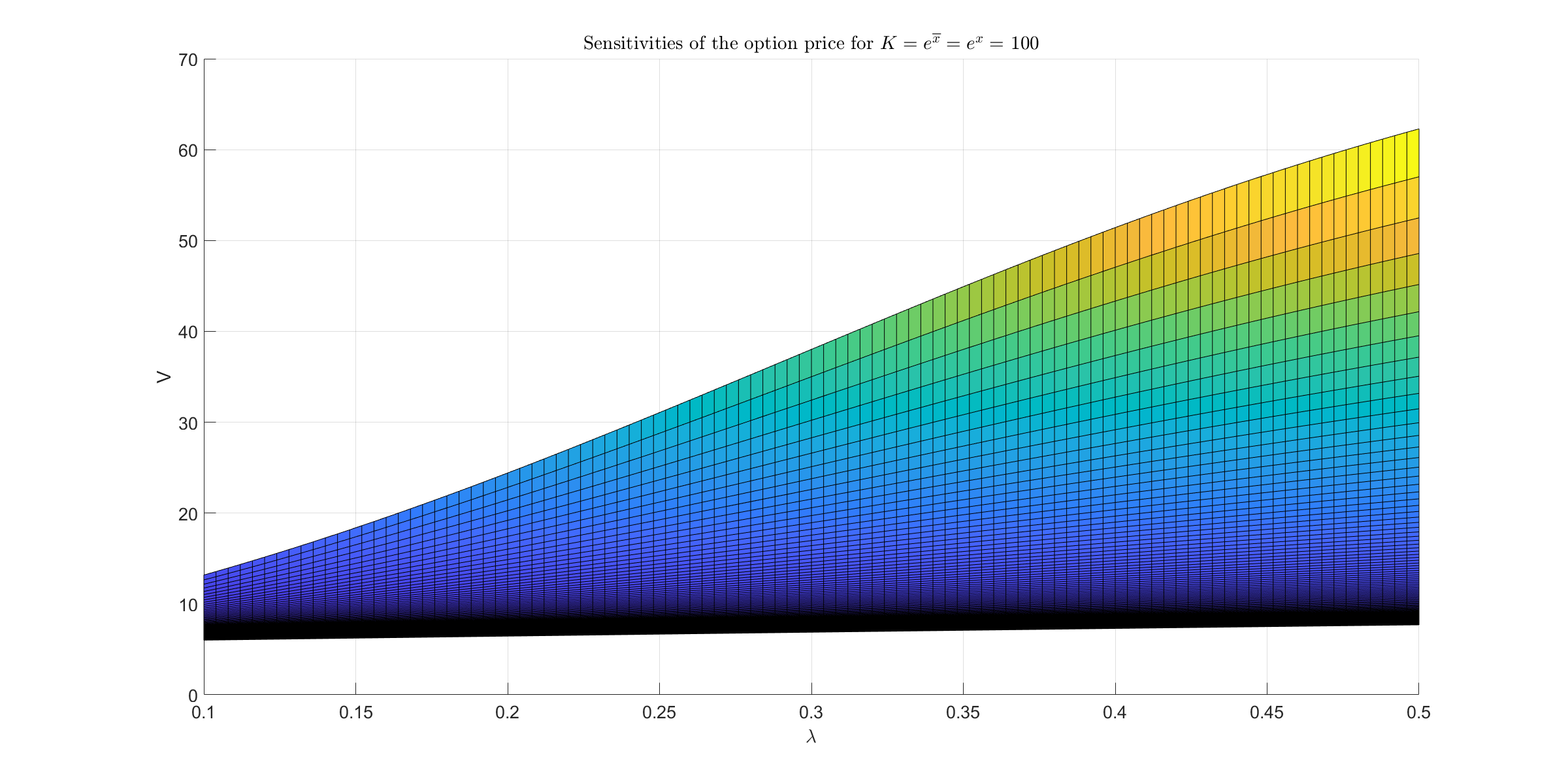}
			\caption{Sensitivities of the option price depending on the jump size $\rho$ and the jump intensity $\lambda$.}
			\label{V_lambda}	
		\end{figure}

	\end{document}